\documentclass[a4paper, 11pt]{scrreprt}

\usepackage[utf8]{inputenc}
\usepackage[USenglish]{babel}
\usepackage[T1]{fontenc}
\usepackage[onehalfspacing]{setspace}
\usepackage{amsmath}
\usepackage{amssymb}
\usepackage{mathtools}
\usepackage{amsthm}
\usepackage{nicefrac}
\usepackage{graphicx}
\usepackage[arrow, matrix, curve]{xy}
\usepackage{extarrows}
\usepackage{stmaryrd}
\usepackage{bm}
\usepackage{tikz-cd}
\usepackage{blindtext}
\usepackage{relsize}
\usetikzlibrary{arrows}
\usepackage{afterpage}
\usepackage{mathdots}
\usepackage{url}

\DeclarePairedDelimiter\abs{\lvert}{\rvert}
\DeclarePairedDelimiterX{\norm}[1]{\lVert}{\rVert}{#1}

\addtokomafont{sectioning}{\rmfamily}

\newtheorem{theorem}{Theorem}[section]
\theoremstyle{definition}
\newtheorem{lemma}[theorem]{Lemma}
\newtheorem{cor}[theorem]{Corollary}
\newtheorem{definition}[theorem]{Definition}
\newtheorem{remark}[theorem]{Remark}
\newtheorem{example}[theorem]{Example}

\newcommand\blankpage{%
    \null
    \thispagestyle{empty}%
    \addtocounter{page}{-1}%
    \newpage}

\newcommand{\id}{\text{id}}
\newcommand{\vnull}{\mathbb{V}_0}
\newcommand{\veins}{\mathbb{V}_1}

\newcommand{\lzi}{\lambda_0(i)}
\newcommand{\lzj}{\lambda_0(j)}
\newcommand{\lzk}{\lambda_0(k)}
\newcommand{\eqpf}{=_{\mathsmaller{\mathcal{F}(X)}}}
\newcommand{\eqsubset}{=_{\mathsmaller{\mathcal{P}(X)}}}
\newcommand{\eqcompset}{=_{\mathsmaller{\mathcal{P}^{\mathsmaller{\mathsmaller{][}}}(X)}}}

\begin{document}
\title{Families of Sets in Constructive Measure Theory}
\author{Max Zeuner \\ Department of Mathematics, Stockholm University \\ \texttt{zeuner@math.su.se}}
\date{}

\maketitle
\addtocounter{page}{-1}

\afterpage{\blankpage}

\tableofcontents
\afterpage{\blankpage}
\chapter{Introduction}
Within Bishop-style constructive mathematics, or constructive analysis more precisely, the development of constructive measure and integration theory is especially interesting and challenging. Since classical measure theory draws heavily on notions of classical set theory and is highly non-constructive in its character, many central concept cannot directly be translated into a constructive setting. Bishop himself was not entirely satisfied with the generality of the measure-theoretic results in chapters 6 and 7 of his seminal book \cite{Bishop67} and developed a new approach to constructive measure theory together with Cheng in \cite{Bishop1972}. This Bishop-Cheng measure theory (BCMT) is also included and expanded in Bishop and Bridges' book \cite{BB85}.

Accordingly, the presentation of set theory is a bit different in the two books in order to meet the demands of the respective measure theories. The set-theoretic operations on the complemented subsets of a set $X$ were changed to behave better with BCMT's function theoretic approach to integration and measure theory, while parts of chapter 3 of \cite{Bishop67} that became obsolete for BCMT were entirely left out in chapter 3 of \cite{BB85}. However, the general treatment of set theory is rather short in both books. This led to confusions about e.g. the status of the powerset in Bishop-style mathematics. It seems that in BCMT the powerset of a set $X$ is itself treated as a set, while in the measure theory of \cite{Bishop67} the powerset is avoided by invoking set-indexed families of subsets instead.

The use of the powerset makes BCMT an impredicative theory and even though BCMT remains an active area of research\footnote{There is for example the recent book on constructive probability theory by Chan, see \cite{Chan2019}, or work on formalizing BCMT in the proof-assistant Coq, see \cite{Semeria2019}.}, it is desirable to have a predicative approach to constructive measure theory. Such an alternative was realized by the point-free, algebraic approach of Coquand and Palmgren in \cite{Coquand2002}, which was further developed by Spitters in \cite{Spitters2005} and \cite{Spitters2006}. This constructive algebraic measure theory is not only predicative, but also conceived by many as technically and conceptually simpler than BCMT.

However, as already mentioned certain impredicativities were generally avoided by Bishop in his '67 book by using set-indexed families instead of the powerset. So, one might hope that by being more explicit about the use of such families, BCMT can be 'predicativized'. Afterall, set-indexed families of sets and subsets are only briefly introduced in chapter 3 of \cite{BB85} and often appear only implicitly later in the book. Recently Petrakis has studied these families more extensively in \cite{Petrakis2018} and \cite{Petrakis2019} and has applied his results to the theory of Bishop spaces. A full-scale account of various kinds of set-indexed families, and their application to different areas of constructive mathematics, will be given in \cite{Petrakis2020}. In this paper we want follow this strategy to show that central parts of BCMT can be done predicatively.

Of course, we have to revisit Bishop's set theory and introduce all the relevant notions before we can turn to measure theory. The paper will thus be structured as follows:
\begin{enumerate}
\item Following Petrakis' work we will introduce an informal language suitable for Bishop's set theory (BST) in the next section. This will allow us to talk in more detail about the fundamental set-theoretic concepts we will need for BCMT.
\item After that we introduce families of sets and subsets, indexed by some set $I$ in BST. In order to use set-indexed families in BCMT we will also introduce families of complemented subsets and partial functions indexed by some set $I$.
\item Having introduced the necessary technical background, we look at the impredicativities in BCMT and refine the notions integration and measure space by the explicit use of set-indexed families in order to introduce what we call pre-integration and pre-measure spaces that are defined predicatively.
\item We give an example of a pre-measure space and show how the pre-integration space of simple functions over a pre-measure space can be defined.
\item Finally, we introduce the pre-integration space of canonically integrable functions over a pre-integration as a predicative version of the complete extension. This shows that the impredicative totality of integrable functions is not needed in order to construct the complete extension.
\end{enumerate}
This paper revisits only some parts of BCMT, so we will sketch a few directions, which further research using pre-integration and pre-measure spaces could take.

\chapter{Bishop's set theory}
Before we can introduce the various forms set-indexed families we need to discuss the fundamental notions of Bishop set theory (BST) necessary for formulating Bishop-Cheng measure theory. Note that even though some things are presented a bit more detailed here than in \cite{BB85}, we still work informally.

\section{Basics}
In this section we will follow \cite{Petrakis2019} to introduce the basic ingredients of BST.\footnote{See \cite{Petrakis2018} for another formulation. The main difference is that there is a universe of functions $\veins$ but no dependent assignment routines.} First we have a primitive notion of definitional equality $:=$ between terms and a primitive notion of ordered pair $(s,t)$ for some terms $s,t$ together with projections $\mathsf{pr}_1(s,t):=s$ and $\mathsf{pr}_2(s,t):=t$. There is also the primitive notion of a \emph{totality} and a primitive totality of natural numbers $\mathbb{N}$. Any other totality $X$ is defined through a membership condition $x\in X$, i.e. a formula expressing the condition an object $x$ has to satisfy in order to belong to $X$. For any totality $X$, an \emph{equality} is a formula $x=_X y$ defined for $x,y\in X$ such that it satisfies the properties of an equivalence relation. For the primitive totality $\mathbb{N}$ we also have a primitive equality $=_{\mathbb{N}}$. A pair $(X,=_X)$, where $X$ is a totality with an equality $=_X$ is a \emph{set} if the member-ship condition expresses a construction. In particular $(\mathbb{N},=_{\mathbb{N}})$ is a set. We will often drop the equality of a set if it is clear from the context.

In BST the notion of a function is not primitive but can be derived from the primitive notion of an \emph{asssignment routine} (or finite routine). For totalities $X,Y$ an assignment routine $\alpha:X\rightsquigarrow Y$ assigns to each element $x\in X$ an element $y\in Y$, in which case we write $\alpha(x):=y$. For any $X$ we have an assignment routine $\id_X:X\rightsquigarrow X$, which assigns each $x\in X$ to itself. If $(X,=_X)$ and $(Y,=_Y)$ are sets, an assignment routine $f:X\rightsquigarrow Y$ is a \emph{function} if
\begin{align*}
\forall x,x'\in X:\;x=_X x'\;\Rightarrow\; f(x)=_Y f(x')
\end{align*}
In this case we write $f:X\rightarrow Y$ and denote the totality of functions  $X\rightarrow Y$ by $\mathbb{F}(X,Y)$. We have an equality $=_{\mathsmaller{\mathbb{F}(X,Y)}}$ on $\mathbb{F}(X,Y)$ which is defined by
\begin{align*}
f=_{\mathsmaller{\mathbb{F}(X,Y)}}g \quad :\Leftrightarrow\quad \forall x\in X:\; f(x)=_Yg(x)
\end{align*}
It easy to check that this satisfies the properties of an equivalence relation and for any two sets $X,Y$ we have the set of functions $(\mathbb{F}(X,Y),=_{\mathsmaller{\mathbb{F}(X,Y)}})$ from $X$ to $Y$. We say a function $e:X\rightarrow Y$ is an embbeding if
\begin{align*}
\forall x,x'\in X:\;x=_X x'\;\Leftrightarrow\; e(x)=_Y e(x')
\end{align*}
In this case we write $e:X\hookrightarrow Y$. We now want to give an important example of a totality that is not a set.

\begin{example}\label{universe}
Let $\vnull$ be the totality of sets. We have an equality on $\vnull$ defined by
\begin{align*}
X=_{\vnull}Y \;:\Leftrightarrow\; \exists f\in\mathbb{F}(X,Y)\;\exists g\in\mathbb{F}(Y,X)\text{ s.t. }f\circ g=_{\mathsmaller{\mathbb{F}(X,X)}}\id_X \;\&\; g\circ f=_{\mathsmaller{\mathbb{F}(Y,Y)}}\id_Y
\end{align*}
In this case we write $(f,g):X=_{\vnull}Y$ However, $(\vnull,=_{\vnull})$ is not a set, since in this case we would have $\vnull\in\vnull$, which is not acceptable from a predicative point of view. There is no general construction to construct a set $X\in\vnull$. Since we still have an equality on $\vnull$, we say that it is a \emph{class}.
\end{example}

Now let $I$ be a set and consider an assignment routine $\lambda_0:I\rightsquigarrow\vnull$. We have a primitive notion of a \emph{dependent assignment routine} over $\lambda_0$, which we will write as
\begin{align*}
\Phi:\bigcurlywedge_{i\in I}\lzi
\end{align*}
and which assigns to each element $i\in I$ an element $\Phi_i\in\lzi$. We have a totality $\mathbb{A}(I,\lambda_0)$ of dependent assignment routines over $\lambda_0$ with equality
\begin{align*}
\Phi=_{\mathsmaller{\mathbb{A}(I,\lambda_0)}}\Psi \;:\Leftrightarrow\;\forall i\in I:\;\Phi_i=_{\lzi}\Psi_i
\end{align*}

Let $X$ be a set, a \emph{subset} of $X$ is a pair $(A,i_A)$, where $A$ is a set and $i_A:A\hookrightarrow X$ is an embedding. We have a totality $\mathcal{P}(X)$ of subsets of $X$, which is not a set since this would require that $\vnull$ is a set. However, we still have an equality $\eqsubset$ that is defined by
\begin{align*}
(A,i_A)\eqsubset(B,i_B)\;:\Leftrightarrow\;\exists f\in\mathbb{F}(A,B)\;\exists g\in\mathbb{F}(B,A):\;&\forall a\in A:\; i_A(a)=_X\big(i_B\circ f\big)(a)\\
\&\;&\forall b\in B:\; i_B(b)=_X\big(i_A\circ g\big)(b)
\end{align*}
i.e. the following diagrams commute
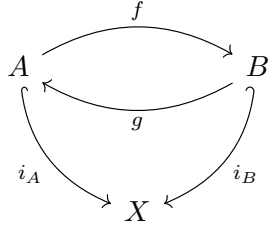
\[
\begin{tikzcd}
A \arrow[rr, "f", bend left]\arrow[rdd, "i_A"', hook, bend right] &&B \arrow[ll, "g", bend left]\arrow[ldd, "i_B",left hook->, bend left] \\
\\
&X
\end{tikzcd}
\]
In this case we say that $f$ and $g$ \emph{witness} the equality of $A$ and $B$ as subsets and write $(f,g):A\eqsubset B$, as we will often drop the the embedding of a subset if it is clear from the context. It is easy to prove that in this case $f$ and $g$ are embeddings and
\begin{align*}
(f,g):A\eqsubset B\;\Rightarrow\;(f,g):A=_{\vnull}B
\end{align*}
The next lemma shows that the witnesses for the equality of subsets are unique. This fact will be usefull later.
\begin{lemma}\label{uniquewitness}
Let $X$ be a set and $(A,i_A)$ and $(B,i_B)$ be subsets of $X$. Let $f,\tilde{f}:A\rightarrow B$ and $g,\tilde{g}:B\rightarrow A$ be functions such that
\begin{itemize}
\item $(f,g):A\eqsubset B$ and
\item $(\tilde{f},\tilde{g}):A\eqsubset B$
\end{itemize}
then $f=_{\mathsmaller{\mathbb{F}(A,B)}}\tilde{f}$ and $g=_{\mathsmaller{\mathbb{F}(B,A)}}\tilde{g}$. 
\end{lemma}
\begin{proof}
Let $a\in A$, then $i_A(a)=_X i_B\big( f(a)\big)$ and $i_A(a)=_X i_B\big( \tilde{f}(a)\big)$, and hence by transitivity $i_B\big( f(a)\big)=_X i_B\big( \tilde{f}(a)\big)$. Since $i_B$ is an embedding we get that $f(a)=_B\tilde{f}(a)$, and thus $f=_{\mathsmaller{\mathbb{F}(A,B)}}\tilde{f}$. By a completely analogous arument we also get that $g=_{\mathsmaller{\mathbb{F}(B,A)}}\tilde{g}$.
\end{proof}
For two subsets $A$ and $B$ of $X$ we say that $A$ is a subset of $B$ if there is a function $f:A\rightarrow B$ such that for all $a\in A$ we have $i_A(a)=_X\big(i_B\circ f\big)(a)$, i.e. if the following diagram commutes
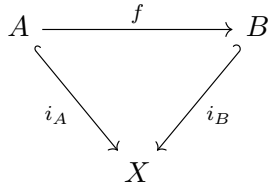
\[
\begin{tikzcd}
A \arrow[rr, "f",]\arrow[rdd, "i_A"', hook] &&B\arrow[ldd, "i_B",left hook->] \\
\\
&X
\end{tikzcd}
\]
In this case we write $f:A\subseteq B$. Again, it is easy to prove that in this case $f$ is an embedding and we get
\begin{align*}
f:A\subseteq B\;\&\;g:B\subseteq A \;\Rightarrow\; (f,g):A\eqsubset B
\end{align*}

When we discuss BCMT later, most of the subsets we will be concerned with are not obtained by constructing some set and then some embedding, but rather by seperating elements out of some existing set. We shall make this precise now. Let $X$ be a set and $P$ a property defined for elements of $X$, which is \emph{extensional} in the sense that it respects the equality on $X$, i.e. we have that
\begin{align*}
\forall x,y\in X:\;x=_X y \;\Rightarrow\; \big(P(x)\;\Rightarrow\;P(y)\big)
\end{align*}
Let $X_P$ be the totality with membership-condition defined by
\begin{align*}
x\in X_P \;:\Leftrightarrow\;x\in X\;\&\;P(x)
\end{align*}
Then we get an equality on $X_P$ induced by by the equality $=_X$ and  the assignment routine $i_{X_P}:X_P\rightsquigarrow X$ given by $x\mapsto x$ is an embedding. We will usually denote the subset $(X_P,i_{X_P})$ of $X$ by $\{x\in X:\;P(x)\}$ and take the embedding as implicitly given. The next example will be of importance for our definition of set-indexed families.

\begin{example}
Consider two sets $X$ and $Y$, the cartesian product $X\times Y$ is the totality defined by the membership-formula
\begin{align*}
z\in X\times Y \;:\Leftrightarrow\;\exists x\in X\;\exists y\in Y\text{ s.t. }\big(z:=(x,y)\big)
\end{align*}
together with the equality
\begin{align*}
z=_{X\times Y}z'\;:\Leftrightarrow\;\mathsf{pr}_1(z)=_X\mathsf{pr}_1(z')\;\&\;\mathsf{pr}_2(z)=_Y\mathsf{pr}_2(z')
\end{align*}
In particular $X\times Y$ is again a set. consider the following property $P$ on $X\times X$:
\begin{align*}
P(z)\;:\Leftrightarrow\;\big(z:=(x,y)\;\Rightarrow\;x=_X y\big)
\end{align*}
This is an extensional property and we call the set $\big({X\times X}\big)_P$ the \emph{diagonal} of $X$ and denote it by
\begin{align*}
D(X):=\;\{(x,y)\in X\times X:\;x=_X y\}
\end{align*}
\end{example}

\begin{definition}
Let $X$ be a set and $(A,i_A)$ and $(B,i_B)$ be subsets of $X$. Let the union $A\cup B$ be the totality defined by the membership-condition
\begin{align*}
c\in A\cup B \;:\Leftrightarrow\;\big(\exists a\in A\text{ s.t. } c:=a\big)\;\text{or}\;\big(\exists b\in B\text{ s.t. } c:=b\big)
\end{align*}
with the equality
\begin{align*}
c=_{A\cup B}c'\;:\Leftrightarrow\;&\big(\exists a,a'\in A\text{ s.t. }c:=a\;\&\; c':=a'\;\&\; a=_A a'\big)\text{ or}\\
&\big(\exists b,b'\in B\text{ s.t. }c:=b\;\&\; c':=b'\;\&\; b=_B b'\big)
\end{align*}
and embedding
\begin{align*}
i_{A\cup B}:A\cup B&\hookrightarrow X \\
c&\mapsto
\begin{cases}
i_A(a),\text{ if } c:=a \text{ for }a\in A \\
i_B(b),\text{ if } c:=b \text{ for }b\in B
\end{cases}
\end{align*}
Let the intersection $A\cap B$ be the set $\{(a,b)\in A\times B:\;i_A(a)=_Xi_B(b)\}$ with the embedding
\begin{align*}
i_{A\cap B}:A\cap B&\hookrightarrow X \\
(a,b)&\mapsto i_A(a)
\end{align*}
It is easy to check that $i_{A\cap B}$ and $i_{A\cup B}$ are indeed embeddings, turning $A\cap B$ and $A\cup B$ into subsets of $X$. Later we will see a different way to define union and intersection that will turn out equivalent up to equality of subsets.
\end{definition}

\section{Complemented subsets and partial functions}
So far, we have seen how to talk about sets, functions and subsets in BST. In BCMT we are however primarily concerned with partial functions and complemented subsets, which we want to introduce in this section. The operations on complemented subsets as presented in \cite{BB85} are different from that in \cite{Bishop67} and we will present the former account in order to be able to use them in BCMT. Before we do that we will briefly introduce partial functions.

Let $X$ and $Y$ be sets. A partial function from $X$ to $Y$ is a triple $(A,i_A,f)$ such that $(A,i_A)$ is subset of $X$ and $f\in\mathbb{F}(A,Y)$. We call $A$ the domain of $f$ and write $f:X\rightharpoonup Y$ if the domain is clear from the context. Let $\mathbb{F}^{\rightharpoonup}(X,Y)$ be the totality of partial functions from $X$ to $Y$ with the following equality. We say two partial functions $(A,i_A,f)$ and $(B,i_B,g)$ are equal if there are functions $\varphi\in\mathbb{F}(A,B)$ and $\psi\in\mathbb{F}(B,A)$ such that the following diagrams commute
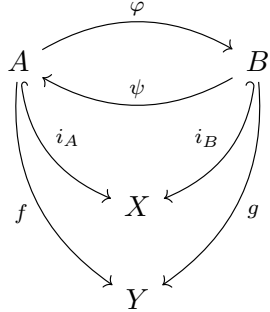
\[
\begin{tikzcd}
A \arrow[rr, "\varphi", bend left]\arrow[rdd, "i_A", hook, bend right]\arrow[rddd, "f"', bend right] &&B \arrow[ll, "\psi"', bend left]\arrow[ldd, "i_B"', left hook->, bend left]\arrow[lddd, "g", bend left] \\
\\
&X \\
&Y
\end{tikzcd}
\]
In this case we will write $(\varphi,\psi):(A,i_A,f)=_{\mathsmaller{\mathbb{F}^{\rightharpoonup}(X,Y)}}(B,i_B,g)$ or just $(\varphi,\psi):f=_{\mathsmaller{\mathbb{F}^{\rightharpoonup}(X,Y)}}g$ if the domains are clear from the context. It follows immediatelly that in this case we also have $(\varphi,\psi):(A,i_A)\eqsubset(B,i_B)$. Using lemma (\ref{uniquewitness}) we can conclude that the witnesses for the equality of partial functions are unique. Since the domain of a partial functions can in principle be any subset of $X$ the totality of partial functions cannot be a set.

Here we will mostly be concerned with real-valued partial functions, i.e. the case where $Y:=\mathbb{R}$. For the totality of real-valued partial functions on a set $X$ we write $\mathcal{F}(X):=\mathbb{F}^{\rightharpoonup}(X,\mathbb{R})$. We want to define the usual (pointwise) operations on $\mathcal{F}(X)$. Let  $(A,i_A,f)$ and $(B,i_B,g)$ be real-valued partial functions and $\ast$ be any of the operations $+,-,\cdot,\wedge,\vee$ on $\mathbb{R}$ (Here $\wedge,\vee$ denote the minimum and maximum of two reals), then we have a partial function $(A\cap B,i_{A\cap B},f\ast g)$ with
\begin{align*}
f\ast g:A\cap B&\rightarrow \mathbb{R}\\
(a,b)&\mapsto f(a)\ast g(b)
\end{align*}
Also for any $a\in\mathbb{R}$ and any $\phi\in\mathbb{F}(\mathbb{R},\mathbb{R})$ we have partial functions $(A,i_A,a\cdot f)$ and $(A,i_A,\phi\circ f)$ that are obtained similarly by pointwise multiplying with $a$ or composing with $\phi$. 

Since negative definitions are often not very useful in constructive mathematics, we  have to use positive counterparts. Complemented subsets are the positive counterpart to the complement of a subset.
\begin{definition}
Let $X$ be a set and let $\neq_X$ be a binary relation defined on $X$. Then $\neq_X$ is said to be an \emph{inequality} or apartness relation on $X$ if the following conditions are satisfied:
\begin{enumerate}
\item $\forall x,y\in X:\; (x=_Xy\;\&\;x\neq_Xy)\;\Rightarrow\;\bot$
\item $\forall x,y\in X:\; x\neq_Xy\;\Rightarrow\;y\neq_Xx$
\item $\forall x,y\in X:\; x\neq_Xy\;\Rightarrow\;\big(\forall z\in X:\; x\neq_Xz\text{ or }y\neq_Xz\big)$
\end{enumerate}
We say that two subsets $(A,i_A)$ and $(B,i_B)$ of $X$ are apart with respect to $\neq_X$ if
\begin{align*}
\forall a\in A\;\forall b\in B:\;i_A(a)\neq_Xi_B(b)
\end{align*}
In this case we write $A][_{\neq_X}B$ or just $A][B$ if the inequality $\neq_X$ is clear from the context. A \emph{complemented subset} of $X$ with respect to $\neq_X$ is a pair of subsets $\bm A:=(A^1,A^0)$ such that $A^1][_{\neq_X}A^0$.
\end{definition}
The notion of an apartness relation is the positive counterpart of the mere negation of the equality $=_X$. From now on let $X$ be equipped with a fixed inequality $\neq_X$. We denote the totality of complemented subsets by $\mathcal{P}^{][}(X)$. Again this totality is obviously not a set. Still, we have an equality defined on $\mathcal{P}^{][}(X)$ by
\begin{align*}
\bm A\eqcompset\bm B \;:\Leftrightarrow\; A^1\eqsubset B^1 \;\&\; A^0\eqsubset B^0
\end{align*}
For complemented subsets $\bm A:=(A^1,A^0)$ and $\bm B:=(B^1,B^0)$ we write $\bm A<\bm B$ if $A^1\subseteq B^1$ and $B^0\subseteq A^0$. We now want to introduce operations of complented union, intersection and substraction as well as complementation. Let $\bm A:=(A^1,A^0)$ and $\bm B:=(B^1,B^0)$ be complemented subsets and define
\begin{itemize}
\item $\bm A\wedge\bm B:=\big(A^1\cap B^1,\;(A^1\cap B^0)\cup(A^0\cap B^1)\cup(A^0\cap B^0)\big)$
\item $\bm A\vee\bm B:=\big((A^1\cap B^0)\cup(A^0\cap B^1)\cup(A^1\cap B^1),\;A^0\cap B^0\big)$
\item $-\bm A:= (A^0,A^1)$ 
\item $\bm A-\bm B := \bm A\wedge (-\bm B)$
\end{itemize}
We can prove commutativity and  associativity of $\wedge$ and $\vee$ as well as the following (with $\bm C$ being a complemented subset):
\begin{enumerate}
\item $--\bm A\eqcompset \bm A$
\item $-(\bm A \wedge \bm B)\eqcompset (-\bm A)\vee (-\bm  B)$
\item $-(\bm A \vee \bm B)\eqcompset (-\bm A)\wedge (-\bm  B)$
\item $\bm A\wedge(\bm B\vee\bm C)\eqcompset(\bm A\wedge\bm B)\vee(\bm A\wedge\bm C)$
\item $\bm A\vee(\bm B\wedge\bm C)\eqcompset(\bm A\vee\bm B)\wedge(\bm A\vee\bm C)$
\end{enumerate}
The poofs are straightforward though tedious. The definitions of $\wedge$ and $\vee$ look a bit complicated and indeed the corresponding definitions in chapter 3 of \cite{Bishop67} appear simpler with the same algebraic laws still being provable. However, the definitions presented here behave much better with the characteristic functions of complemented subsets. Let $\bm 2:=\{0,1\}$, for a complemented subset $\bm A:=(A^1,A^0)$ we define its characteristic function by
\begin{align*}
\chi_{\bm A}:\;&A^1\cup A^0\rightarrow\bm 2 \\
&x\mapsto\begin{cases}1,\;x\in A^1 \\
0,\;x\in A^0
\end{cases}
\end{align*}
So, if we interpret $\bm 2$ as a subset on $\mathbb{R}$ the characteristic function of a complemented subset is an element of $\mathcal{F}(X)$ and it is easy to check that
\begin{itemize}
\item $\chi_{\bm A\wedge\bm B}\eqpf\chi_{\bm A}\wedge\chi_{\bm B}$
\item $\chi_{\bm A\vee\bm B}\eqpf\chi_{\bm A}\vee\chi_{\bm B}$
\item $\chi_{-\bm A}\eqpf 1-\chi_{\bm A}$
\end{itemize}
where the operations on the characteristic functions are the pointwise minimum and maximum that we defined above and composition with the function $(1-\_):\mathbb{R}\rightarrow\mathbb{R}$. With this the above properties of the complemented operations can also be proven more easily. We conclude the section with a final remark on partial functions and inequalities.

\begin{remark}
We have a canonical inequality on $\mathbb{R}$ that is given by
\begin{align*}
x\neq_{\mathbb{R}}y \;:\Leftrightarrow\;\abs{x-y}>0
\end{align*}
If $X$ is equipped with an inequality $\neq_X$ as well, we take $\mathcal{F}(X)$ to be the totality of strongly extensional partial functions, i.e. for any $(A,i_A,f)$ we assume that
\begin{align*}
\forall a,b\in A:\; f(a)\neq_{\mathbb{R}}f(b) \;\Rightarrow\; i_A(a)\neq_Xi_A(b)
\end{align*}
\end{remark}

\section{On double series of real numbers}\label{reals}
At the end of this section we want to consider some aspects of constructive analysis that will be needed for our discussion of BCMT later. One important point is that converging or Cauchy sequences of reals are taken to be modulated in order to avoid the axiom of countable choice. We will also prove a result that is implicitly used in \cite{BB85} for crucial results of BCMT but doesnt't seem to be explicitly stated in the literature on constructive mathematics. 

We will follow the presentation in \cite{Schwichtenberg2006}. A real number is a pair $x=\big((a_n)_{n\in\mathbb{N}},M)$ consisting of a sequences of rational numbers $(a_n)_{n\in\mathbb{N}}$ and $M:\mathbb{N}\rightarrow\mathbb{N}$ is strictly increasing such that
\begin{align*}
\forall p\in\mathbb{N}\;\forall n,m\geq M(p):\; \abs{a_n - a_m}\leq\frac{1}{2^p}
\end{align*}
Two reals $x=\big((a_n)_{n\in\mathbb{N}},M)$ and $y=\big((b_n)_{n\in\mathbb{N}},N)$ are equal, which we write as $x=_{\mathbb{R}}y$, if
\begin{align*}
\forall p\in\mathbb{N}:\;\abs{a_{M(p+1)}-b_{N(p+1)}}\leq \frac{1}{2^p}
\end{align*}
Note that this is equivalent to
\begin{align*}
\forall p\in\mathbb{N}\;\exists n_0\in\mathbb{N}\;\forall n\geq n_0:\;\abs{a_n-b_n}\leq\frac{1}{2^p}
\end{align*}
Using this second characterisation of equality it is not hard to show that $=_{\mathbb{R}}$ is indeed an equality on the totality $\mathbb{R}$. The usual operations like addition, multiplication and absolute value and the order realtions can be defined, and the usual basic properties, like the triangle inequality, can be derived.

\begin{definition}
Let $(x_n)_{n\in\mathbb{N}}$ be a sequence of reals, $x\in\mathbb{R}$ and $M:\mathbb{N}\rightarrow\mathbb{N}$ strictly increasing. We say that $(x_n)_{n\in\mathbb{N}}$ is a Cauchy-sequence with modulus $M$ if
\begin{align*}
\forall p\in\mathbb{N}\;\forall n,m\geq M(p):\; \abs{x_n - x_m}\leq\frac{1}{2^p}
\end{align*}
Furthermore, we say that $(x_n)_{n\in\mathbb{N}}$ converges to $x$ with modulus $M$ if
\begin{align*}
\forall p\in\mathbb{N}\;\forall n\geq M(p):\; \abs{x_n-x}\leq\frac{1}{2^p}
\end{align*}
\end{definition}
Note, that if we use notation like $\lim_{n\rightarrow\infty}x_n=x$ or say that some sequence or series converges we implicitly assume the existence of a corresponding modulus of convergence. 

Finally, for our discussion we will need some facts about the convergence of double series. Reorderings of infinite series have been studied constructively in \cite{Berger2009} and \cite{Berger2013}, but we are concerned with a slightly different problem. We want to show that for some fixed bijection from $\mathbb{N}$ to $\mathbb{N}\times\mathbb{N}$ we can rearrange a double series into a single one if it converges absolutely. Consider the following enumeration of $\mathbb{N}\times\mathbb{N}$:
\[ 
\begin{matrix}
(1,1)^{\;\bm 1} &(1,2)^{\;\bm 3} &(1,3)^{\;\bm 6} &\iddots\\
(2,1)^{\;\bm 2} &(2,2)^{\;\bm 5} &(2,3)^{\;\bm 9}\\
(3,1)^{\;\bm 4} &(3,2)^{\;\bm 8} \\
(4,1)^{\;\bm 7}
\end{matrix}
\]
Using this pattern as an enumeration of  $\mathbb{N}_0\times\mathbb{N}_0$ starting with zero, we would get a bijection that is quite conveniently expressed by the following function\footnote{This is why this particular bijection is often used to code pairs in computability theory.}
\begin{align*}
c:\;&\mathbb{N}_0\times\mathbb{N}_0\rightarrow\mathbb{N}_0\\
&(x,y)\mapsto\frac{(x+y)(x+y+1)}{2}+y
\end{align*}
That means that our enumeration of $\mathbb{N}\times\mathbb{N}$ corresponds to a bijection $\varphi:\mathbb{N}\rightarrow\mathbb{N}\times\mathbb{N}$ with inverse
\begin{align*}
\varphi^{-1}(n,m)=c(n-1,m-1)+1
\end{align*}
Using this particular bijection we can state the following lemma.

\begin{lemma}\label{doubleseries}
Let $\big(x_{nk}\big)_{n,k\in\mathbb{N}}$ be a sequence of sequences in $\mathbb{R}$, i.e. an element of $\mathbb{F}(\mathbb{N},\mathbb{F}(\mathbb{N},\mathbb{R}))$, and let $\big(y_n\big)_{n\in\mathbb{N}}$ be the sequence defined by
\begin{align*}
y_m:=x_{\mathsf{pr}_1\big(\varphi(m)\big)\;\mathsf{pr}_2\big(\varphi(m)\big)}
\end{align*}
i.e. $x_{nk}=y_{\varphi^{-1}(n,k)}$. Then $\sum_n\sum_k x_{nk}$ converges absolutely if and only if $\sum_m y_m$ converges absolutely, in which case we have that
\begin{align*}
\sum_{n=1}^\infty\sum_{k=1}^\infty x_{nk}=\sum_{m=1}^\infty y_m
\end{align*}
\end{lemma}

\begin{proof}
First assume that $\sum_n\sum_k \abs{x_{nk}}$ exists and for the moment let us also assume that $x_{nk}\geq 0$ for all $n,k\in\mathbb{N}$. Assume that $\big(\sum_{k=1}^Nx_{nk}\big)_{N\in\mathbb{N}}$ converges to $\ell_n:=\sum_k x_{nk}$ with modulus $M_n$ and that $\big(\sum_{n=1}^N\sum_k x_{nk}\big)_{N\in\mathbb{N}}$ converges to $\ell:=\sum_n\sum_k x_{nk}$ with modulus $M$. We have to construct a modulus $M':\mathbb{N}\rightarrow\mathbb{N}$ s.t. for $p\in\mathbb{N}$ and $N\geq M'(p)$ we have 
\begin{align*}
0\leq\ell-\sum_{m=1}^Ny_m\leq\frac{1}{2^p}
\end{align*}
First, consider $N\geq M(p+1)$, then
\begin{align*}
0\leq \ell-\sum_{n=1}^N\ell_n\leq\frac{1}{2^{(p+1)}}
\end{align*}
Now, consider $N\geq M_n(M(p+1)+p+1)$, then
\begin{align*}
0\leq\ell_n-\sum_{k=1}^{N}x_{nk}\leq\frac{1}{2^{M(p+1)}}\frac{1}{2^{(p+1)}}\leq\frac{1}{M(p+1)}\frac{1}{2^{(p+1)}}
\end{align*}
Hence
\begin{align*}
0&\leq\ell- \sum_{n=1}^{M(p+1)}\;\;\sum_{k=1}^{M_n(M(p+1)+p+1)} x_{nk}\\[1em]
&=\ell-\sum_{n=1}^{M(p+1)}\ell_n+\sum_{n=1}^{M(p+1)}\Big(\ell_n-\sum_{k=1}^{M_n(M(p+1)+p+1)} x_{nk}\Big)\\ &\leq\frac{1}{2^p}
\end{align*}
Observe that by our choice of $\varphi$ for any $N\in\mathbb{N}$ and $m,n\leq N$ we have that $\varphi^{-1}(m,n)\leq\varphi^{-1}(N,N)$. If we define $M_{\varphi}:\mathbb{N}\rightarrow\mathbb{N}$ by $M_{\varphi}(N):=\varphi^{-1}(N,N)$, then $M_{\varphi}$ is strictly increasing. We can now define $M':\mathbb{N}\rightarrow\mathbb{N}$ by
\begin{align*}
M'(p):=M_{\varphi}\Big( M(p+1)\vee\big(\bigvee_{n=1}^{M(p+1)}M_n(M(p+1)+p+1)\big)\Big)
\end{align*}
and for $p\in\mathbb{N}$ and $N\geq M'(p)$ we have that
\begin{align*}
\sum_{m=1}^Ny_m= \sum_{n=1}^{M(p+1)}\;\;\sum_{k=1}^{M_n(M(p+1)+p+1)} x_{nk}+c(N,p)
\end{align*}
where $c(N,p)\geq 0$ is the finite sum of the $x_{nk}$ s.t. $n>M(p+1)$ or $k>M_n(M(p+1)+p+1)$ but $\varphi^{-1}(n,k)\leq N$. It follows that
\begin{align*}
0\leq\ell-\sum_{m=1}^{N}y_m\leq\frac{1}{2^p}
\end{align*}

Now for the general case. Assume that $\big(\sum_{k=1}^Nx_{nk}\big)_{N\in\mathbb{N}}$ converges to $\ell_n:=\sum_k x_{nk}$ with modulus $M_n$ and that $\big(\sum_{n=1}^N\sum_k x_{nk}\big)_{N\in\mathbb{N}}$ converges to $\ell:=\sum_n\sum_k x_{nk}$ with modulus $M$. Furthermore, assume that $\big(\sum_{k=1}^N\abs{x_{nk}}\big)_{N\in\mathbb{N}}$ converges to $\tilde{\ell}_n:=\sum_k \abs{x_{nk}}$ with modulus $\tilde{M}_n$ and that $\big(\sum_{n=1}^N\sum_k \abs{x_{nk}}\big)_{N\in\mathbb{N}}$ converges to $\tilde{\ell}:=\sum_n\sum_k \abs{x_{nk}}$ with modulus $\tilde{M}$. We have already shown that $\sum_m\abs{y_m}=\tilde{\ell}$ and may assume that this series converges with some modulus $\tilde{M}'$. Without loss of generality we can also assume that $M\geq \tilde{M}'$ because otherwise we can replace $M$ by $M\vee \tilde{M}'$. Again, we get for $p\in\mathbb{N}$ that
\begin{align*}
\abs*{\ell- \sum_{n=1}^{M(p+1)}\;\;\sum_{k=1}^{M_n(M(p+1)+p+1)} x_{nk}}\leq\frac{1}{2^p}
\end{align*}
Thus, if we define $M':\mathbb{N}\rightarrow\mathbb{N}$ by
\begin{align*}
M'(p):=M_{\varphi}\Big( M(p+2)\vee\big(\bigvee_{n=1}^{M(p+2)}M_n(M(p+2)+p+2)\big)\Big)
\end{align*}
and take $N\geq M'(p)$ we get
\begin{align*}
\abs*{\ell-\sum_{m=1}^Ny_m}\leq\frac{1}{2^{(p+1)}}+\abs{c(N,p)}
\end{align*}
and
\begin{align*}
\abs{c(N,p)}\leq\sum_{\substack{(n,k)\text{ s.t. }\varphi^{-1}(n,k)\leq N,\\ 
n>M(p+2)\text{ or }k>M_n(M(p+2)+p+2)}}\abs{x_{nk}}\leq\sum_{m=M(p+2)}^\infty\abs{y_m}\leq\sum_{m=\tilde{M}'(p+2)}^\infty\abs{y_m}\leq\frac{1}{2^{(p+1)}}
\end{align*}
which finishes the first part of the proof. For the converse direction assume that $\tilde{\ell}:=\sum_m\abs{y_m}$ exists and that this series converges with modulus $\tilde{M}$ and let $\ell:=\sum_my_m$ and assume that this series converges with modulus $M$. 

We first have to show that for any $n\in\mathbb{N}$, $\tilde{\ell}_n:=\sum_k\abs{x_{nk}}$ exists, so fix $n\in\mathbb{N}$ and let $N_0\leq N_1$, then
\begin{align*}
\sum_{k=N_0}^{N_1}\abs{x_{nk}}=\sum_{m=N_0}^{N_1}\abs{y_{\varphi^{-1}(n,m)}}\leq\sum_{m=N_0}^\infty\abs{y_m}
\end{align*}
since $\varphi^{-1}(n,m)\geq m$ for all $m\in\mathbb{N}$. It follows that $\big(\sum_{k=1}^N\abs{x_{nk}}\big)_{N\in\mathbb{N}}$ is a Cauchy-sequence and hence we may assume that  $\tilde{\ell}_n:=\sum_k\abs{x_{nk}}$ exists and the series converges with modulus $\tilde{M}$ and that  $\ell_n:=\sum_k{x_{nk}}$ exists and the series converges with modulus $M$. We next have to show that $\sum_n\tilde{\ell}_n$ exists, so let $N_0\leq N_1$, then again
\begin{align*}
\sum_{n=N_0}^{N_1}\tilde{\ell}_n\leq\sum_{m=N_0}^\infty\abs{y_m}
\end{align*}
which shows that $\big(\sum_{n=1}^N\tilde{\ell}_n\big)_{N\in\mathbb{N}}$ is a Cauchy-sequence and it remains to check that $\sum_n\ell_n=\ell$. So fix $N\in\mathbb{N}$, let $p\in\mathbb{N}$ and define $N_{n,p}:=M_n(N+p)$ for $1\leq n\leq N$. Note that $N_{n,p}\geq N$ for all $1\leq n\leq N$ and
\begin{align*}
\abs*{\ell-\sum_{n=1}^N\ell_n}\leq\underbrace{\sum_{n=1}^N\;\;\abs*{\sum_{k=1}^{N_{n,p}}x_{nk}-\ell_n}}_{\leq2^{-p}}+\underbrace{\abs*{\sum_{n=1}^N\sum_{k=1}^{N_{n,p}}x_{nk}-\ell}}_{(*)}
\end{align*}
and in particular
\begin{align*}
(*)\leq\underbrace{\sum_{n=1}^N\sum_{k=1}^{N_{n,p}}\abs{x_{nk}-y_{\varphi^{-1}(n,k)}}}_{=0}+\;\;\sum_{\substack{(n,k)\text{ s.t. } n>N \\ \text{or }k>N_{n,p}\geq N}}\abs{y_{\varphi^{-1}(n,k)}}\leq\sum_{m=N}^\infty\abs{y_m}
\end{align*}
Since $p\in\mathbb{N}$ was chosen arbitrarily we get that $\abs*{\ell-\sum_{n=1}^N\ell_n}\leq\sum_{m=N}^\infty\abs{y_m}$, which finishes the proof.
\end{proof}

\afterpage{\blankpage}
\chapter{Set indexed families}
We have seen that a lot of important notions of BST and BCMT give rise to totalities that, although equipped with a canonical equality, are not sets or can at least not be treated as such from a predicative point of view. However, we are still able to speak about these totalities by invoking so-called \emph{set-indexed families}. We introduce the various kinds of set-indexed families in this section.
\section{Families of sets}
As a first example of a totality that is not a set we considered the universe $\vnull$. It is thus natural to introduce families of sets before introducing families of subsets, complemented subsets and partial functions. However, when applying set-indexed families to BCMT we won't be concerned families of sets. Therefore, the account given here (following \cite{Petrakis2019}) shall be rather brief and is just included for the sake of completeness.
\begin{definition}\label{famset}
Let $I$ be a set, an \emph{$I$-family of sets} is a pair $\Lambda:=(\lambda_0,\lambda_1)$, where $\lambda_0:I\rightsquigarrow\mathbb{V}_0$ and $\lambda_1:\bigcurlywedge_{i,j\in D(I)}\mathbb{F}(\lzi,\lzj)$ is a dependent assignment routine such that for every $(i,j)\in D(I)$ and $\lambda_1(i,j):=\lambda_{ij}$ we have that $\lambda_{ii}:=\id_{\lambda_0(i)}$, and for every $i,j,k\in I$, satisfying $i=_I j$ and $j=_I k$, the following diagram commutes
\[
\begin{tikzcd}
\lzi\arrow[rd, "\lambda_{ik}"]\arrow[d,"\lambda_{ij}"'] \\
\lzj\arrow[r, "\lambda_{jk}"'] & \lzk
\end{tikzcd}
\]
We call $I$ the \emph{index set} of the family $\Lambda$, the function $\lambda_{ij}$ the \emph{transport} function from $\lambda_0(i)$ to $\lambda_0(j)$, and the dependent assignment routine $\lambda_1$ the \emph{modulus of function-likeness} of $\lambda_0$. We say that $\Lambda$ is an $I$-\emph{set} of sets if\footnote{Note that the left to right direction always holds, since for any family of sets we have $i=_I j\Rightarrow(\lambda_{ij},\lambda_{ji}):\lzi=_{\vnull}\lzj$.}
\begin{align*}
\forall i,j\in I:\; i=_I j\;\Leftrightarrow\;\lzi=_{\vnull}\lzj
\end{align*}
\end{definition}

\begin{definition}
Let $I$ be a set and $\Lambda:=(\lambda_0,\lambda_1)$ and $M:=(\mu_0,\mu_1)$ be $I$-families of sets. A \emph{family map} from $\Lambda$ to $M$ is a dependent assignement routine
\begin{align*}
\Psi:\bigcurlywedge_{i\in I}\mathbb{F}\big(\lzi,\mu_0(i)\big)
\end{align*}
such that for every $i=_I j$ the following diagram commutes
\[
\begin{tikzcd}
\lzi \arrow[d, "\Psi_i"']\arrow[r, "\lambda_{ij}"] &\lzj\arrow[d, "\Psi_j"] \\
\mu_0(i)\arrow[r, "\mu_{ij}"'] &\mu_0(j)
\end{tikzcd}
\]
In this case we write $\Psi:\Lambda\Rightarrow M$. We denote by $\mathtt{Map}_I(\Lambda,M)$ the totality of family maps from $\Lambda$ to $M$, which is equipped with the equality
\begin{align*}
\Psi=_{\mathsmaller{\mathtt{Map}_I(\Lambda,M)}}\Xi \;:\Leftrightarrow\;\forall i\in I:\;\Psi_i=_{\mathsmaller{\mathbb{F}(\lzi,\mu_0(i))}}\Xi_i
\end{align*}
If $N:=(\nu_0,\nu_1)$ is an $I$-family as well, then for $\Psi:\Lambda\Rightarrow M$ and $\Xi:M\Rightarrow N$ the composition $\Psi\circ\Xi:\Lambda\Rightarrow N$ is defined as $(\Psi\circ\Xi)_i:=\Psi_i\circ\Xi_i$. The identity family-map from $\Lambda$ to $\Lambda$ is the dependent assignment routine $\text{Id}_{\Lambda}:\bigcurlywedge_{i\in I}\mathbb{F}(\lzi,\lzi)$ defined by $\big(\text{Id}_{\Lambda}\big)_i:=\id_{\lzi}$.
\end{definition}

Over a family of sets we can define the exterior union and the dependent product, but since we won't use these notions in our discussion of BCMT, we just state the definitions.

\begin{definition}\label{extunion}
Let $\Lambda:=(\lambda_0,\lambda_1)$ be an $I$-family of sets. The \emph{exterior union}, or \emph{disjoint union} $\sum_{i\in I}\lambda_0(i)$ of $\Lambda$ is defined by
\begin{align*}
w\in\sum_{i\in I}\lambda_0(i)\; &:\Leftrightarrow\; \exists i\in I\;\exists x\in\lambda_0(i)\text{ s.t. }w:=(i,x) \\
(i,x)=_{\sum_{i\in I}\lambda_0(i)}(j,y) \; &:\Leftrightarrow\;i=_I j \;\text{and}\; \lambda_{ij}(x)=_{\lambda_0(j)}y
\end{align*}
\end{definition}

\begin{definition}\label{depproduct}
Let $\Lambda:=(\lambda_0,\lambda_1)$ be an $I$-family of sets. The \emph{dependent product} $\prod_{i\in I}\lzi$ is defined by
\begin{align*}
\Phi\in \prod_{i\in I}\lzi \;&:\Leftrightarrow\; \Phi:\bigcurlywedge_{i\in I}\lzi \text{ s.t. }
\forall (i,j)\in D(I):\;\lambda_{ij}(\Phi_i)=_{\lzj}\Phi_j \\
\Phi=_{\prod_{i\in I}\lzi}\Psi \;&:\Leftrightarrow\; \forall i\in I:\; \Phi_i=_{\lzi}\Psi_i
\end{align*}
\end{definition}

\section{Families of subsets}
We next introduce families of subsets that we will use a lot in our discussion of BCMT.
\begin{definition}\label{famsubset}
Let $X$ and $I$ be sets. A \emph{family of subsets} indexed by $I$ is a triple $\lambda:=(\lambda_0,\mathcal{E},\lambda_1)$, where $\lambda_0:I\rightsquigarrow\vnull$, and $\mathcal{E}:\bigcurlywedge_{i\in I} \mathbb{F}(\lzi,X)$, such that, for every $i\in I$, we have that $\mathcal{E}(i):=\varepsilon_i$ is an embedding of $\lambda_0(i)$ into $X$ (i.e. $(\lambda_0(i),\varepsilon_i)$ is a subset of $X$). Moreover, $\lambda_1:\bigcurlywedge_{(i,j)\in D(I)}\mathbb{F}(\lzi,\lzj)$ is called the \emph{modulus of function-likeness} of $\lambda_0$, and for every $i\in I$ it satisfies $\lambda_{ii}:=\id_{\lambda_0(i)}$, while for every $(i,j)\in D(I)$ it satisfies $(\lambda_{ij},\lambda_{ji}):\lambda_0(i)=_{\mathcal{P}(X)}\lambda_0(j)$, i.e. the following inner diagrams commute
\[
\begin{tikzcd}
\lambda_0(i) \arrow[rr, "\lambda_{ij}", bend left]\arrow[rdd, "\varepsilon_i"', hook, bend right] &&\lambda_0(j) \arrow[ll, "\lambda_{ji}", bend left]\arrow[ldd, "\varepsilon_j",left hook->, bend left] \\
\\
&X
\end{tikzcd}
\]
We say that $\Lambda$ is an $I$-\emph{set} of subsets if
\begin{align*}
\forall i,j\in I:\; i=_I j\;\Leftrightarrow\;\lzi\eqsubset\lzj
\end{align*}
Let $\lambda:=(\lambda_0,\mathcal{E},\lambda_1)$ and $\mu:=(\mu_0,E,\mu_1)$ be $I$-families of subsets of $X$. A \emph{family of subsets-map} is a dependent assignment routine $\Psi:\bigcurlywedge_{i\in I}\mathbb{F}\big(\lzi,\mu_0(i)\big)$ such that for every $i\in I$ the following diagram commutes
\[
\begin{tikzcd}
\lzi \arrow[rr, "\Psi_i"]\arrow[rd, "\varepsilon_i"', hook] &&\mu_0(i)\arrow[ld, "e_i", left hook->] \\
&X
\end{tikzcd}
\]
i.e. $\Psi_i:\lzi\subseteq\mu_0(i)$ for all $i\in I$. In this case we write $\Psi:\lambda\Rightarrow_X \mu$. We denote by $\mathtt{Map}_I^X(\lambda,\mu)$ the totality of family maps from $\lambda$ to $\mu$, which is equipped with the equality
\begin{align*}
\Psi=_{\mathsmaller{\mathtt{Map}_I^X(\lambda,\mu)}}\Xi \;:\Leftrightarrow\;\forall i\in I:\;\Psi_i=_{\mathsmaller{\mathbb{F}(\lzi,\mu_0(i))}}\Xi_i
\end{align*}
If $\nu:=(\nu_0,E',\nu_1)$ is an $I$-family of subsets of $X$ as well, then for $\Psi:\lambda\Rightarrow_X \mu$ and $\Xi:\mu\Rightarrow_X \nu$ the composition $\Psi\circ\Xi:\lambda\Rightarrow_X \nu$ is defined as $(\Psi\circ\Xi)_i:=\Psi_i\circ\Theta_i$. The identity family of subsets-map from $\lambda$ to $\lambda$ is the dependent assignment routine $\text{Id}_{\lambda}:\bigcurlywedge_{i\in I}\mathbb{F}(\lzi,\lzi)$ defined by $\big(\text{Id}_{\lambda}\big)_i:=\id_{\lzi}$.
\end{definition}

\begin{remark}
It is easy to check that for any $I$-family of subsets of $X$, $\lambda:=(\lambda_0,\mathcal{E},\lambda_1)$, we get a family of sets $\Lambda_{\lambda}:=(\lambda_0,\lambda_1)$. From the definition we get that for $i=_Ij$ and $j=_I k$ the following diagrams commute
\[
\begin{tikzcd}
\lzj\arrow[r, "\lambda_{jk}"]\arrow[rd, "\varepsilon_j", hook] &\lzk\arrow[d, "\varepsilon_k", hook]\\
\lzi\arrow[u, "\lambda_{ij}"]\arrow[r, "\varepsilon_i", hook] &X
\end{tikzcd}\quad\quad\quad
\begin{tikzcd}
\lzj\arrow[d, "\lambda_{ji}"]\arrow[rd, "\varepsilon_j", hook] &\lzk\arrow[d, "\varepsilon_k", hook]\arrow[l, "\lambda_{kj}"']\\
\lzi\arrow[r, "\varepsilon_i", hook] &X
\end{tikzcd}
\]
It follows that $(\lambda_{ij}\circ\lambda_{jk},\lambda_{kj}\circ\lambda_{ji}):\lzi\eqsubset\lzk$. But since we have $i=_I k$ we get that $(\lambda_{ik}\lambda_{ki}):\lzi\eqsubset\lzk$ holds as well and we get from lemma (\ref{uniquewitness}) that in particular $\lambda_{ij}\circ\lambda_{jk}=_{\mathsmaller{\mathbb{F}(\lzi,\lzk)}}\lambda_{ik}$, i.e. the following diagram commutes
\[
\begin{tikzcd}
\lzi\arrow[rd, "\lambda_{ik}"]\arrow[d,"\lambda_{ij}"'] \\
\lzj\arrow[r, "\lambda_{jk}"'] & \lzk
\end{tikzcd}
\]
Furthermore, a family of subsets-map $\Psi:\lambda\Rightarrow_X \mu$ is also a family map $\Psi:\Lambda_{\lambda}\Rightarrow M_{\mu}$, since for $i=_I j$ the following inner diagrams commute
\[
\begin{tikzcd}
\lzi\arrow[rr, "\lambda_{ij}"]\arrow[dd, "\Psi_i"']\arrow[rd, "\varepsilon_i", hook] &&\lzj\arrow[dd, "\Psi_j"]\arrow[ld, "\varepsilon_j"', left hook->] \\
&X \\
\mu_0(i)\arrow[rr, "\mu_{ij}"']\arrow[ru, "e_i", hook] &&\mu_0(j)\arrow[lu, "e_j"', left hook->]
\end{tikzcd}
\]
For $x\in\lzi$ we thus get
\begin{align*}
\big(e_j\circ\mu_{ij}\circ\Psi_i\big)(x)&=_X\big(e_i\circ\Psi_i\big)(x)\\
&=_X\varepsilon_i(x)\\
&=_X\big(\varepsilon_j\circ\lambda_{ij}\big)(x)\\
&=_X\big(e_j\circ\Psi_j\circ\lambda_{ij}\big)(x)
\end{align*}
and since $e_j$ is an embedding we get that $\big(\mu_{ij}\circ\Psi_i\big)(x)=_{\mu_0(j)}\big(\Psi_j\circ\lambda_{ij}\big)(x)$.
\end{remark}
Next, we want to introduce the notion of a subfamily that will play an important role in redefining the notions of measure and integration space in BCMT.

\begin{definition}
Let $\lambda:=(\lambda_0,\mathcal{E},\lambda_1)$ be an $I$-family of subsets of $X$ and $J$ be a set with a function $h:J\rightarrow I$. Then by $(\lambda\circ h):=(\lambda_0\circ h, \mathcal{E}\circ h,\lambda_1\circ h)$ we denote the $J$ family of subsets that is given by composing all (dependent) assignment routines with $h$. Now let $\mu:=(\mu_0,E,\mu_1)$ be a $J$-family of subsets of $X$, we say that $\mu$ is a \emph{subfamily} of $\lambda$ via $h$ if there are family of subset-maps $\Psi:\mu\Rightarrow_X (\lambda\circ h)$ and $\Psi^{-1}:(\lambda\circ h)\Rightarrow_X \mu$ such that $\Psi\circ\Psi^{-1}=_{\mathsmaller{\mathtt{Map}_J^X(\mu,\mu)}}\text{Id}_{\mu}$ and $\Psi^{-1}\circ\Psi=_{\mathsmaller{\mathtt{Map}_J^X((\lambda\circ h),(\lambda\circ h))}}\text{Id}_{(\lambda\circ h)}$, i.e. for each $j\in J$ the following diagrams commute
\[
\begin{tikzcd}
\mu_0(j)\arrow[rd, "e_j"', bend right, hook]\arrow[r, "\Psi_j"']\arrow[rr, "\id_{\mu_0(j)}", bend left] &\lambda_0\big(h(j)\big)\arrow[r, "\Psi_j^{-1}"']\arrow[d, "\varepsilon_{h(j)}", hook] &\mu_0(j)\arrow[ld, "e_j", bend left, left hook->]\\
&X
\end{tikzcd}\quad\quad
\begin{tikzcd}
\lambda_0\big(h(j)\big)\arrow[rd, "\varepsilon_{h(j)}"', bend right, hook]\arrow[r, "\Psi_j^{-1}"']\arrow[rr, "\id_{\lambda_0(h(j))}", bend left] &\mu_0(j)\arrow[d, "e_j", hook]\arrow[r, "\Psi_j"'] &\lambda_0\big(h(j)\big)\arrow[ld, "\varepsilon_{h(j)}", bend left, left hook->]\\
&X
\end{tikzcd}
\]
\end{definition}

\begin{remark}\label{subfamembedding}
If in the above setting $\mu$ is a $J$-set of subsets, then $h:J\rightarrow I$ is an embedding. To see this, assume that we have $j,j'\in J$ s.t. $h(j)=_I h(j')$. Then the following inner diagrams commute
\[
\begin{tikzcd}
&\lambda_0\big(h(j)\big)\arrow[rr, "\lambda_{h(j)h(j')}"]\arrow[rd, "\varepsilon_{h(j)}", hook] &&\lambda_0\big(h(j')\big)\arrow[rd, "\Psi_{j'}^{-1}"]\arrow[ld, "\varepsilon_{h(j')}"', left hook->]\\
\mu_0(j)\arrow[ru, "\Psi_j"]\arrow[rr, "e_j", hook] &&X &&\mu_0(j')\arrow[ll, "e_{j'}"', left hook->]\arrow[ld, "\Psi_{j'}"]\\
&\lambda_0\big(h(j)\big)\arrow[ru, "\varepsilon_{h(j)}"', hook]\arrow[lu, "\Psi_{j}^{-1}"] &&\lambda_0\big(h(j')\big)\arrow[lu, "\varepsilon_{h(j')}", left hook->]\arrow[ll, "\lambda_{h(j')h(j)}"]
\end{tikzcd}
\]
It follows that $\mu_0(j)\eqsubset\mu_0(j')$ and thus $j=_J j'$ since $\mu$ is a $J$-set of subsets.
\end{remark}

Finally we can introduce the appropriate notions union and intersection of a family of subsets, corresponding to the exterior union and dependent product for families of sets.

\begin{definition}
Let $\lambda:=(\lambda_0,\mathcal{E},\lambda_1)$ be an $I$-family of subsets of $X$. The \emph{interior union} of $\lambda$ is the totality $\bigcup_{i\in I}\lambda_0(i)$ defined by
\begin{align*}
z\in\bigcup_{i\in I}\lambda_0(i) \; :\Leftrightarrow\;\exists i\in I\;\exists x\in\lambda_0(i)\text{ s.t. }z:=(i,x)
\end{align*}
Let the assignment routine $\varepsilon:\bigcup_{i\in I}\lambda_0(i)\rightsquigarrow X$ be defined by $\varepsilon(i,x):=\varepsilon_i(x)$. The equality on $\bigcup_{i\in I}\lambda_0(i)$ is defined by
\begin{align*}
(i,x)=_{\bigcup_{i\in I}\lambda_0(i)}(j,y) \; :\Leftrightarrow\;\varepsilon(i,x)=_X\varepsilon(j,y)
\end{align*}
Then $\varepsilon$ is an embedding, turning $\bigcup_{i\in I}\lambda_0(i)$ into a subset of $X$.
The \emph{intersection} of $\lambda$ is the totality $\bigcap_{i\in I}\lambda_0(i)$ defined by
\begin{align*}
\Psi\in \bigcap_{i\in I}\lzi \;&:\Leftrightarrow\; \Psi:\bigcurlywedge_{i\in I}\lzi \text{ s.t. }
\forall i,j\in I:\;\varepsilon_i(\Psi_i)=_X \varepsilon_j(\Psi_j) \\
\Phi=_{\bigcap_{i\in I}\lzi}\Psi \;&:\Leftrightarrow\; \forall i\in I:\; \Phi_i=_{\lzi}\Psi_i
\end{align*}
For $i_0\in I$ the assignment routine
\begin{align*}
\epsilon:\bigcap_{i\in I}\lzi &\hookrightarrow X\\
\Psi &\mapsto \varepsilon_{i_0}(\Psi_{i_0})
\end{align*}
is an embedding (and as a function in $\mathbb{F}\big(\bigcap_{i\in I}\lzi, X\big)$ independent of $i_0$, turning $\bigcap_{i\in I}\lzi$ into a subset of $X$.
\end{definition}

\begin{example}
Let $(A,i_A)$ and $(B,i_B)$ be subsets of $X$ let $\lambda:=(\lambda_0,\mathcal{E},\lambda_1)$ be the $\bm 2$-family of subsets given by 
\begin{itemize}
\item $\lambda_0(0):=A$ and $\lambda_0(1):=B$
\item $\varepsilon_0:=i_A$ and $\varepsilon_1:=i_B$
\item $\lambda_{00}:=\id_A$ and $\lambda_{11}:=\id_B$
\end{itemize}
then it is easy to verify
\begin{align*}
\bigcap_{i\in \bm 2}\lzi\eqsubset A\cap B \quad\text{ and }\quad \bigcup_{i\in\bm 2}\lzi\eqsubset A\cup B
\end{align*}
\end{example}

\begin{remark}
Talking about intersections and unions of subsets includes a lot of data which we will ignore in large parts later in our discussion of BCMT. In particular we will not write elements of unions as pairs but rather identify them with their image in $X$. Similarly, we won't write elements of intersections as dependent assignment routines but treat them like elements of $X$.
\end{remark}

\section{Families of complemented subsets}
We now turn to set-indexed families of complemented subsets. Just as a complemented subset consists of two subsets that are apart, a family of complemented subsets consists of two families of subsets that are apart at every index. We follow \cite{Petrakis2019b} in this section.

\begin{definition}
Let $X$ be a set with a fixed inequality $\neq_X$ and $I$ be a set. An $I$-\emph{family of complemented subsets} of $X$ is a sextuple $\bm{\lambda}:=(\lambda_0^1,\mathcal{E}^1,\lambda_1^1,\lambda_0^0,\mathcal{E}^0,\lambda_1^0)$ where $\lambda^1:=(\lambda_0^1,\mathcal{E}^1,\lambda_1^1)$ and $\lambda^0:=(\lambda_0^0,\mathcal{E}^0,\lambda_1^0)$ are both $I$-families of subsets of $X$ such that
\begin{align*}
\forall i\in I:\; \lambda_0^1(i)\; ][\;\lambda_0^0(i)
\end{align*}
We denote by $\bm\lambda_0(i):=\big(\lambda_0^1(i),\lambda_0^0(i)\big)$ the complemented subset associated to $i\in I$. We say that $\bm\lambda$ is an $I$-\emph{set} of complemented subsets if
\begin{align*}
\forall i,j\in I:\; i=_I j \;\Leftrightarrow\; \bm\lambda_0(i)\eqcompset\bm\lambda_0(j)
\end{align*}
If $\bm\nu:=(\nu_0^1,E^1,\nu_1^1,\nu_0^0,E^0,\nu_1^0)$ is an $I$-family of complemented subsets as well, then a family of complemented subsets-map from $\bm\lambda$ to $\bm\nu$ is a pair $\bm\Psi:=(\Psi^1,\Psi^0)$ such that $\Psi^1:\lambda^1\Rightarrow_X\nu^1$ and $\Psi^0:\lambda^0\Rightarrow_X\nu^0$. We denote by $\mathtt{Map}_I^{X,\; ][}(\bm\lambda,\bm\nu)$ the totality of family of complemented subsets-maps from $\bm\lambda$ to $\bm\nu$, equipped with the equality of componentwise equality of family of subsets-maps. Concatenation of family of complemented subsets-maps and the identity map are also defined componentwise in the canonical way.
\end{definition}

To talk about measure-spaces in BCMT predicatively we also need the notion of a subfamily of complemented subsets.

\begin{definition}
Let $\bm{\lambda}:=(\lambda_0^1,\mathcal{E}^1,\lambda_1^1,\lambda_0^0,\mathcal{E}^0,\lambda_1^0)$ be an $I$-family of complemented subsets of $X$ and $J$ be a set with a function $h:J\rightarrow I$. Then by $(\bm\lambda\circ h)$ we denote the $J$-family of complemented subsets that is given by composing all (dependent) assignment routines with $h$. Now let $\bm\nu:=(\nu_0^1,E^1,\nu_1^1,\nu_0^0,E^0,\nu_1^0)$ be a $J$-family of complemented subsets of $X$, we say that $\bm\nu$ is a subfamily of $\bm\lambda$ via $h$ if there are family of complemented subset-maps $\bm\Psi$ from $\bm\nu$ to $(\bm\lambda\circ h)$ and $\bm\Psi^{-1}$ from $(\bm\lambda\circ h)$ to $\bm\nu$ such that $\bm\Psi\circ\bm\Psi^{-1}=_{\mathsmaller{\mathtt{Map}_J^{X,\;][}(\bm\nu,\bm\nu)}}\text{Id}_{\bm\nu}$ and $\bm\Psi^{-1}\circ\bm\Psi=_{\mathsmaller{\mathtt{Map}_J^{X,\;][}((\bm\lambda\circ h),(\bm\lambda\circ h))}}\text{Id}_{(\bm\lambda\circ h)}$
\end{definition}

By remark (\ref{subfamembedding}) it follows directly from the previous definitions that if in the above setting $\bm\nu$ is $J$-set of complemented subsets, we have an embedding $h:J\hookrightarrow I$.

At the end of this section we want to give an example of a canonical family of complemented subsets that can be constructed for any inhabited set, the family of detachable subsets of a set. Classically, the detachable subsets are in a one-to-one correspondence with the subsets of $X$.
\begin{definition}
Let $X$ be an inhabited set and let $\bm 2:=\{0,1\}$. Let $\neq_X$ be the following inequality on $X$:
\begin{align*}
x\neq_X y \;:\Leftrightarrow\; \exists f\in\mathbb{F}(X,\bm 2) \text{ s.t. } f(x)\neq f(y)
\end{align*}
Let $\bm\delta=(\delta_0^1,\mathcal{E}^1,\delta_1^1,\delta_0^0,\mathcal{E}^0,\delta_1^0)$ be the $\mathbb{F}(X,\bm 2)$-family of complemented subset of $(X,\neq_X)$ with $\mathbb{F}(X,\bm 2)$-families of subsets $\delta^1=(\delta_0^1,\mathcal{E}^1,\delta_1^1)$ and $\delta^0=(\delta_0^0,\mathcal{E}^0,\delta_1^0)$ given by the following data:
\begin{itemize}
\item For $f\in\mathbb{F}(X,\bm 2)$ we have that $\delta_0^1(f):=\{x\in X:\;f(x)=_{\bm 2}1\}$ \\ and $\delta_0^0(f):=\{x\in X:\;f(x)=_{\bm 2}0\}$
\item The embeddings $\varepsilon_f^1:\delta_0^1(f)\hookrightarrow X$ and $\varepsilon_f^0:\delta_0^0(f)\hookrightarrow X$ are all defined by $x\mapsto x$ for any $f\in\mathbb{F}(X,\bm 2)$.
\item The transport maps $\delta_{fg}^1:\delta_0^1(f)\rightarrow\delta_0^1(g)$ and $\delta_{fg}^0:\delta_0^0(f)\rightarrow\delta_0^0(g)$ are all defined by $x\mapsto x$ for any $f=_{\mathsmaller{\mathbb{F}(X,\bm 2)}}g$.
\end{itemize}
For $f\in\mathbb{F}(X,\bm 2)$ we define $\bm\delta_0(f):=\big(\delta_0^1(f),\delta_0^0(f)\big)$ to be the associated \emph{detachable complemented subset}.
\end{definition}

\begin{remark}
It is not hard to check that $\neq_X$ actually defines an appartness relation on $X$ and that $\bm\delta$ actually defines a $\mathbb{F}(X,\bm 2)$-family of complemented subsets w.r.t. this inequality. Moreover, we have that $\delta_0^1(f)\cup\delta_0^0(f)\eqsubset X$ and $\chi_{\bm\delta_0(f)}=_{\mathsmaller{\mathbb{F}(X,\bm 2)}}f$ for any $f\in\mathbb{F}(X,\bm 2)$. Using this, one can immediately see that 
\begin{align*}
\forall f,g\in\mathbb{F}(X,\bm 2):\;\; f=_{\mathsmaller{\mathbb{F}(X,\bm 2)}}g \;\Leftrightarrow\; \bm\delta_0(f)\eqcompset \bm\delta_0(g)
\end{align*}
i.e. that $\bm\delta$ is actually a set of complemented subsets.
\end{remark}

\section{Families of partial functions}
Finally, we introduce families of partial functions. Since a partial function consists of a subset (its domain) and a function on that subset, a family of partial functions consists of a family of subsets (the family of domains) and a dependent assignment routine, assigning to each index a function on the respective domain.

\begin{definition}\label{fampf}
Let $X, Y$ and $I$ be sets. An $I$-\emph{family of partial functions} from $X$ to $Y$ is a quadruple $\bm\Lambda:=(\lambda_0,\mathcal{E},\lambda_1,F)$, where
\begin{itemize}
\item $\lambda_0:I\rightsquigarrow\vnull$
\item $\mathcal{E}:\bigcurlywedge_{i\in I} \mathbb{F}(\lzi,X)$, such that, for every $i\in I$, we have that $\mathcal{E}(i):=\varepsilon_i$ is an embedding of $\lambda_0(i)$ into $X$
\item $\lambda_1:\bigcurlywedge_{i,j\in D(I)}\mathbb{F}(\lzi,\lzj)$ is called the \emph{modulus of function-likeness} of $\lambda_0$, and for every $i\in I$ it satisfies $\lambda_{ii}:=\id_{\lambda_0(i)}$, while for every $(i,j)\in D(I)$ it satisfies $(\lambda_{ij},\lambda_{ji}):\lambda_0(i)=_{\mathcal{P}(X)}\lambda_0(j)$
\item $F:\bigcurlywedge_{i\in I}\mathbb{F}(\lambda_0(i),Y)$  such that, for every $i\in I$, we have that $F(i):=f_i$ and for $(i,j)\in D(I)$ and $x\in\lzi$ we have that $f_i(x)=_Y \big(f_j\circ\lambda_{ij}\big)(x)$.
\end{itemize}
i.e. for $(i,j)\in D(I)$ the following diagrams commute
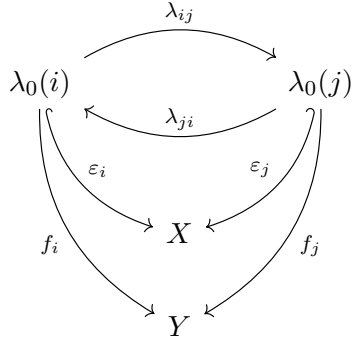
\[
\begin{tikzcd}
\lambda_0(i) \arrow[rr, "\lambda_{ij}", bend left]\arrow[rdd, "\varepsilon_i", hook, bend right]\arrow[rddd, "f_i"', bend right] &&\lambda_0(j) \arrow[ll, "\lambda_{ji}"', bend left]\arrow[ldd, "\varepsilon_j"', left hook->, bend left]\arrow[lddd, "f_j", bend left] \\
\\
&X \\
&Y
\end{tikzcd}
\]
We say that $\bm\Lambda$ is an $I$-\emph{set} of partial functions if
\begin{align*}
\forall i,j\in I:\; i=_I j \;\Leftrightarrow\; f_i=_{\mathsmaller{\mathbb{F}^{\rightharpoonup}(X,Y)}}f_j
\end{align*}
If $\bm M:=(\mu_0, E,\mu_1,G)$ is an $I$-family of partial functions from $X$ to $Y$ as well, a family of partial functions-map from $\bm\Lambda$ to $\bm M$ is a dependent assignment routine $\Psi:\bigcurlywedge_{i\in I}\mathbb{F}(\lzi,\mu_0(i))$ s.t. for every $i\in I$ the following diagrams commute
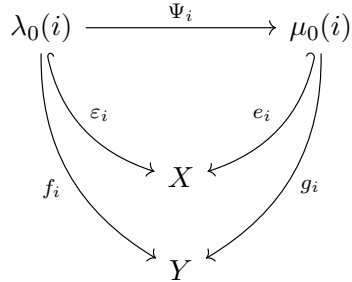
\[
\begin{tikzcd}
\lzi\arrow[rr, "\Psi_i"]\arrow[rdd, "\varepsilon_i", hook, bend right]\arrow[rddd, "f_i"', bend right] &&\mu_0(i)\arrow[ldd, "e_i"', left hook->, bend left]\arrow[lddd, "g_i", bend left] \\
\\
&X \\
&Y
\end{tikzcd}
\]
We denote by $\mathtt{Map}_I^{X\rightharpoonup Y}(\bm\Lambda,\bm M)$ the totality of family of partial functions-maps from $\bm\Lambda$ to $\bm M$, which is equipped with the equality
\begin{align*}
\Psi=_{\mathsmaller{\mathtt{Map}_I^{X\rightharpoonup Y}(\bm\Lambda,\bm M)}}\Xi \;:\Leftrightarrow\;\forall i\in I:\;\Psi_i=_{\mathsmaller{\mathbb{F}(\lzi,\mu_0(i))}}\Xi_i
\end{align*}
Composition of family of partial functions-maps and the identity map are defined just like for family of subsets-maps.
\end{definition}

\begin{remark}
Note that any $I$-family of partial functions $\bm\Lambda:=(\lambda_0,\mathcal{E},\lambda_1,F)$ we get an induced $I$-family of subsets $(\lambda_0,\mathcal{E},\lambda_1)$, which we may call the family of domains. It follows directly from the definitions that a family of partial functions-map is also a family of subsets-map between the respective families of domains.
\end{remark}

\begin{definition}\label{subfamilies pf}
Let $\bm\Lambda:=(\lambda_0,\mathcal{E},\lambda_1,F)$ be an $I$-family of partial functions from $X$ to $I$ and $J$ be a set with a function $h:J\rightarrow I$. Then by $(\bm\Lambda\circ h)$ we denote the $J$ family of partial functions  that is given by composing all (dependent) assignment routines with $h$. Now let $\bm M:=(\mu_0,E,\mu_1;G)$ be a $J$-family of partial functions, we say that $\mu$ is a subfamily of $\lambda$ via $h$ if there are family of partial functions-maps $\Psi$ from $\bm M$ to $(\bm\Lambda\circ h)$ and $\Psi^{-1}$ from $(\bm\Lambda\circ h)$ to $\bm M$, such that $\Psi\circ\Psi^{-1}=_{\mathsmaller{\mathtt{Map}_I^{X\rightharpoonup Y}(\bm M,\bm M)}}\text{Id}_{\bm M}$ and $\Psi^{-1}\circ\Psi=_{\mathsmaller{\mathtt{Map}_I^{X\rightharpoonup Y}((\bm\Lambda\circ h),(\bm\Lambda\circ h))}}\text{Id}_{(\bm\Lambda\circ h)}$, i.e. for each $j\in J$ the following diagrams commute.
\[
\begin{tikzcd}
&Y\\
\mu_0(j)\arrow[rd, "e_j"', bend right, hook]\arrow[r, "\Psi_j"']\arrow[ru, "g_j", bend left] &\lambda_0\big(h(j)\big)\arrow[r, "\Psi_j^{-1}"']\arrow[d, "\varepsilon_{h(j)}", hook]\arrow[u, "f_{h(j)}"'] &\mu_0(j)\arrow[ld, "e_j", bend left, left hook->]\arrow[lu, "g_j"', bend right]\\
&X
\end{tikzcd}\quad\quad
\begin{tikzcd} &X\\
\lambda_0\big(h(j)\big)\arrow[rd, "\varepsilon_{h(j)}"', bend right, hook]\arrow[r, "\Psi_j^{-1}"']\arrow[ru, "f_{h(j)}", bend left] &\mu_0(j)\arrow[d, "e_j", hook]\arrow[r, "\Psi_j"']\arrow[u, "g_j"'] &\lambda_0\big(h(j)\big)\arrow[ld, "\varepsilon_{h(j)}", bend left, left hook->]\arrow[lu, "f_{h(j)}"', bend right]\\
&X
\end{tikzcd}
\]
Note that this means $(\Psi_j,\Psi_j^{-1}):\;g_j=_{\mathsmaller{\mathbb{F}^{\rightharpoonup}(X,Y)}}f_{h(j)}$ for any $j\in J$ .
\end{definition}

\begin{remark}
If in the above setting $\bm M$ is a $J$-set of partial functions, we get that $h:J\hookrightarrow I$ is actually an embedding. This follows from the fact that any family of partial functions-map is a family of subsets-map on the family of domains, which allows us to use the proof given in remark (\ref{subfamembedding}).
\end{remark}

\chapter{Measure theory}
So far, we have considered some of the fundamental notions of BST and we have introduced the corresponding notions of set-indexed families. In this section we want to apply these tools to BCMT. Not only will our reformulation be predicative but we also can develop central parts of BCMT entirely without the axiom of countable choice.

\section{On the impredicative character of BCMT}
The basic notions of BCMT are integration and measure spaces. Their definitions in \cite{BB85} are already somewhat problematic from a predicative point of view. Moreover, the complete extension of an integration space, which plays a crucial role in BCMT, takes the class of integrable functions to a be set. As we will see below, this is not acceptable predicatively, a fact which has been pointed out repeatedely in the literature. In fact, the development of algebraic constructive measure theory mentioned in the introduction specifically addresses this problem and proposes a solution by giving a new point-free approach. 

However, we are interested in a solution that still works within BCMT or BISH more generally. In the next section we will take an in-depth look at the impredicativities present in BCMT and single out three particular problems that a predicative account needs to address. Two of those problems can be resolved by refining the definition of an integration and a measure space. This will lead us to the notions of a pre-integration and pre-measure space that, among other things, make explicit use of set-indexed families. The third problem concerns the integrable function and our proposal for its solution will be postponed to the last section of this capter.

\subsection{Three problems}
Consider the definition of a measure space as given in \cite[p. 282]{BB85}:
\begin{quote}
A \emph{measure space} is a triple $(X,M,\mu)$ consisting of a nonvoid set $X$ with an inequality $\neq$, a set $M$ of complemented sets in $X$, and a mapping $\mu$ of $M$ into $\mathbb{R}^{0+}$, such that the following properties hold. [...]
\end{quote}
and likewise that of an integration space \cite[p. 217]{BB85}:
\begin{quote}
A triple $(X,L,I)$ is an \emph{integration space} if $X$ is a nonvoid set with an inequality $\neq$, $L$ a subset of $\mathcal{F}(X)$, and $I$ is a mapping of $L$ into $\mathbb{R}$ such that the following properties hold. [...]
\end{quote}
The  question now arises, how we can speak about or define sets of complemented subsets and real-valued partial functions if the totalities $\mathcal{P}^{][}(X)$ and $\mathcal{F}(X)$ are not sets. Defining these sets by seperating them out of the totalities of complemented subsets or partial functions by some extensional property is not possible. However, at least the phrase "\emph{$L$ a subset of $\mathcal{F}(X)$}" seems to suggest such a reading.

This first problem is easily resolved using the tools introduced in the previous section. We defined sets of complemented subsets or partial functions to be certain kinds of set-indexed families. It is safe to assume that Bishop had these notions of sets of complemented subsets or partial functions in mind, but did not explicitly mention the respective index sets in order to make BCMT more accessible to classically trained mathematicians.

We want to make sure that we work absolutely predicatively and thus introduce the required indexing explicitely. A measure space is thus actually a quadruple $(X,I,\bm\lambda,\mu)$ where $X$ is an inhabited set with a fixed inequality $\neq_X$, $I$ is some index-set, $\bm\lambda$ is an $I$-set of complemented subsets of $X$ and $\mu$ is the measure, such that certain conditions hold. Similarly, an integration space is thus actually a quadruple $(X,I,\bm\Lambda,\int)$ where $X$ is an inhabited set with a fixed inequality $\neq_X$, $I$ is some index-set, $\bm\Lambda$ is an $I$-set of real valued partial functions of $X$ and $\int$ is the integral, such that certain conditions hold.

The second problem is a bit more serious. Consider the second defining property of a measure space in \cite[p. 282]{BB85}:
\begin{quote}
(10.1.2) If $A$ and $A\wedge B$ belong to $M$, then so does $A-B$, and $\mu(A)=\mu(A\wedge B)+\mu(A-B)$.
\end{quote}
Formally, with the indexing being made explicit, this condition looks as follows:
\begin{align*}
&\forall i\in I\;\forall \bm B\in\mathcal{P}^{][}(X):\; \bigg(\exists j\in I \text{ s.t. } \bm\lambda_0(j)=\bm\lambda_0(i)\wedge \bm B\bigg) \; \\
&\Rightarrow\;\bigg(\exists k\in I \text{ s.t. } \bm\lambda_0(k)=\bm\lambda_0(i)-\bm B 
\quad\&\quad \mu(\bm\lambda_0(i))=\mu(\bm\lambda_0(j))+\mu(\bm\lambda_0(k))\bigg)
\end{align*}
This means that the definition of measure space contains universal quantification over $\mathcal{P}^{][}(X)$, a totality which is not a set. From a predicative point of view this is not an acceptable definition. In his '67 book Bishop was much more cautious regarding this problem, giving the follong definition of measure space (see \cite[p.183]{Bishop67}):
\begin{quote}
[...] Let $\mathfrak{F}$ be any family of complemented subsets of $X$ [...]
Let $\mathfrak{M}$ be a subfamily of $\mathfrak{F}$ closed under finite unions, intersections, and differences. Let the function $\mu:\mathfrak{M}\rightarrow\mathbb{R}^{0+}$ satisfy the following conditions [...]
\end{quote}
Here a measure space takes the form $(X,I,\bm\lambda,J,\bm\nu,\mu)$, where $\bm\lambda$ is an $I$-family and $\bm\nu$ is a $J$-family of complemented subsets s.t. $\bm\lambda$ is a subfamily of $\bm\nu$. The measure $\mu$ is defined on the `smaller' family $\bm\lambda$ and the `larger' family $\bm\nu$ works as a substitute for the totality $\mathcal{P}^{][}(X)$. The troubling defining condition then takes the form:
\begin{quote}
If $A\in\mathfrak{M}$, if $B\in\mathfrak{F}$, and if $A\cap B\in\mathfrak{M}$, then $A-B\in\mathfrak{M}$ and $\mu(A)=\mu(A\cap B)+\mu(A-B)$
\end{quote}
This means that we can replace quantification over $\mathcal{P}^{][}(X)$ by quantification over the larger index set $J$.

Before we put everything we discussed above together to give a precise and predicative definition of integration and measure space we want to mention the third problem that is also the most serious one. Consider the definition of an integrable function over some integration space $(X,L,I)$ in \cite[p. 222]{BB85}:
\begin{quote}
An element $f$ of $\mathcal{F}(X)$ is an \emph{integrable function} if there exists a sequence $(f_n)$ of functions in $L$ such that $\sum_n I(\abs{f_n})$ converges and, and $f(x)=\sum_n{f_n(x)}$ whenever $\sum_n\abs{f_n(x)}$ converges. The sequence $(f_n)$ is called a \emph{representation} of $f$ [...] in $L$. We write $L_1$ for the class of integrable functions.
\end{quote}
It has been pointed out that this class $L_1$ is in general not a set, since its membership condition explicitly involves the totality of partial functions. However, the complete extension of an integration space contains $L_1$ as its set of partial functions, on which the extended integral is defined. Our strategy to resolve this problem is to replace $L_1$ by something smaller that actually constitutes a set of partial functions and show that working with that smaller family is actually sufficient to construct the complete extension. 

\subsection{Pre-measure and pre-integration spaces}

Before tackling the third problem mentioned in the previous section, we have to give precise definitions of the notions of integration and measure space that we want to work with. We argued that sets of complemented subsets and partial functions should explicitly be treated as set-indexed families. This forces us to make a few choices on how to express the defining conditions of integration and measure spaces. In BCMT a measure is a function on a set of complemented subsets and an integral is function on a set of partial functions. Since we are concerned with set-indexed families rather than simple sets we have to make precise what kind of functions measures and integrals actually are. Also, the definitions of measure and integration spaces involve certain closedness conditions that we want to rewrite in a way that allows us to avoid the use of choice principles.

In his paper "Mathematics as a numerical language" Bishop makes a suggestion on how to formalize the measure theory of his '67 book:
\begin{quote}
A \emph{measure space} is a family $\mathcal{M}\equiv\{A_t\}_{t\in T}$ of complemented subsets of a set $X$ [...], a map $\mu:T\rightarrow\mathbb{R}^{0+}$ and an additional structure as follows: [...] If $t$ and $s$ are in $T$, there exists an element $s\vee t$ of $T$ such that $A_{s\vee t}<A_t\cup A_s$. Similarly, there exist operations $\wedge$ and $\sim$ on $T$, corresponding to the set-theoretic operations $\cap$ and $-$.
\begin{flushright}
-\hspace{0,2cm}\cite[p. 67]{Bishop1970}
\end{flushright} 
\end{quote}
Translating this to BCMT, there are two main ideas here to consider:
\begin{enumerate}
\item The measure of a measure space is a function defined on the index-set of the set of complemented susets and likewise is the integral a function on the index-set of the set of partial functions.
\item Set theoretic operations under which the complemented subsets of a measure space are closed are introduced via operations on the index-set corresponding to those set-theoretic operation. The same applies to the function-theoretic operations under which the partial functions of an integration space are closed.
\end{enumerate}

Putting everything together we arrive at what will be called pre-integration and pre-measure spaces. For the remainder of this paper we want to show that important parts of BCMT can be recovered using these predicative notions rather the original measure and integration spaces as defined in \cite{BB85}. We start by introducing pre-integration spaces.

\begin{definition}\label{preintegration space}
Let $X$ be an inhabited set with inequality $\neq_X$, $I$ a set, $\bm\Lambda=(\lambda_0,\lambda_1,\mathcal{E},F)$ an $I$-set of real-valued partial functions and $\int:I\rightarrow\mathbb{R}$ a function. Furthermore, assume that we have assignment routines
\begin{align*}
\_\cdot\_:\mathbb{R}\times I\rightsquigarrow I \\
\_+\_:I\times I\rightsquigarrow I \\
\abs{\_}:I\rightsquigarrow I \\
\wedge_1:I\rightsquigarrow I
\end{align*}
Then $(X,I,\bm\Lambda,\int )$ is called a \emph{pre-integration} space if the following conditions hold
\begin{labeling}{(PIS4)}
\item[(PIS1)]$\forall i,j\in I\;\forall a,b\in\mathbb{R}$ we have		
			\begin{itemize}
			\item $f_{a\cdot i }\eqpf a f_i$
			\item $f_{i+j}\eqpf f_i+f_j$
			\item $f_{\abs{i}}\eqpf\abs{f_i}$
			\item $f_{\wedge_1(i)}\eqpf f_i\wedge 1$
			\end{itemize}
			and we have that $\int (a\cdot i +b\cdot j)=_{\mathbb{R}}a\int i + b\int j$
\item[(PIS2)]$\forall i\in I\;\forall\alpha\in\mathbb{F}(\mathbb{N},I)$, if
			\begin{itemize}
			\item $\forall m\in\mathbb{N}:\; f_{\alpha(m)}\geq 0$ and 
			\item $\ell:=\lim_{n\rightarrow\infty}\int \sum_{k=1}^n \alpha(k)=\sum_{k=1}^\infty\int\alpha(k) $ exists and $\ell<\int i$
			\end{itemize}
			then there is $x\in\lzi\cap\big(\bigcap_{n\in\mathbb{N}}\lambda_0(\alpha(n))\big)$ s.t. ${\ell':=\sum_{k=1}^\infty f_{\alpha(k)}(x)}$ exists and $\ell'< f_i(x)$.
\item[(PIS3)] $\exists i\in I$ s.t. $\int i=_{\mathbb{R}}1$
\item[(PIS4)] $\forall i\in I$, for $\alpha,\beta\in\mathbb{F}(\mathbb{N},I)$ defined by $\alpha(m)=m\cdot (\wedge_1(m^{-1}\cdot i)) \text{ and } \beta(m)=m^{-1}\cdot(\wedge_1(m\cdot\abs{i}))$ we have that $\ell:=\lim_{n\rightarrow\infty}\int \alpha(n),\;\ell':=\lim_{n\rightarrow\infty}\int \beta(n)$ exist and $\ell=_{\mathbb{R}}\int i$ and $\ell'=_{\mathbb{R}}0$.
\end{labeling}
\end{definition}

\begin{remark}
From (PIS1) and the fact that $\bm\Lambda$ is an $I$-set of partial functions it follows that the assignment routines $\cdot,+,\abs{\_},\wedge_1$ are functions. Also note that we can define operations $\vee,\wedge:I\rightarrow I$ by $i\vee j := j +\frac{1}{2}\cdot(i-j+\abs{i-j})$ and $i\wedge j:= -\big( (-i)\vee(-j)\big)$ for $i,j\in I$ such that $f_{i\vee j}\eqpf f_i\vee f_j$ and $f_{i\wedge j}\eqpf f_i\wedge f_j$.
\end{remark}

\begin{definition}\label{premeasure space}
Let $X$ be an inhabited set with an apartness-relation $\neq_X$, $I,J$ sets, $\bm{\lambda}=(\lambda_0^1,\lambda_1^1,\mathcal{E}^1,\lambda_0^0,\lambda_1^0,\mathcal{E}^0)$ an $I$-set and $\bm\nu=(\nu_0^1,\nu_1^1,E^1,\nu_0^0,\nu_1^0,E^0)$ a $J$-set of complemented subsets of $X$ s.t. $\bm\lambda$ is a subfamily of $\bm\nu$ (i.e. we have an embedding ${h:I\hookrightarrow J}$ and family maps $\bm\Psi,\bm\Psi^{-1}$ that we take as implicitly given) and $\mu:I\rightarrow\mathbb{R}_{\geq 0}$ a function. Furthermore, assume that we have assignment routines $\wedge:J\times J\rightsquigarrow J$, $\vee:J\times J\rightsquigarrow J$ and $\sim:J\rightsquigarrow J$, as well as $\wedge:I\times I\rightsquigarrow I$, $\vee:I\times I\rightsquigarrow I$ and $\sim:I\times I\rightsquigarrow I$ s.t. for all $i,j\in I$ we have
\begin{itemize}
\item $h(i\wedge j)=_J h(i)\wedge h(j)$
\item $h(i\vee j)=_J h(i)\vee h(j)$
\item $h(i\sim j)=_J h(i)\wedge \sim h(j)$
\end{itemize}
Then $(X,I,\bm\lambda,J,\bm\nu,\mu)$ is a \emph{pre-measure space} if the following conditions hold:
\begin{labeling}{(PMS4)}
\item[(PMS1)] $\forall i,j\in J$ we have
			 \begin{itemize}
			 \item $\bm\nu_0(i\wedge j)\eqcompset\bm\nu_0(i)\wedge\bm\nu_0(j)$
			 \item $\bm\nu_0(i\vee j)\eqcompset\bm\nu_0(i)\vee\bm\nu_0(j)$
			 \item $\bm\nu_0(\sim i)\eqcompset -\bm\nu_0(i)$
			 \end{itemize}
			 and  for $i,j\in I$ we have that $\mu(i)+(j)=_{\mathbb{R}}\mu(i\vee j)+\mu(i\wedge j)$.
\item[(PMS2)] $\forall i\in I\;\forall j\in J:$ If there is a $k\in I$ s.t. $h(k)=_J h(i)\wedge j$, then there exist $l\in I$ s.t. $h(l)=_J h(i)\wedge\sim j$ and $\mu(i)=_{\mathbb{R}}\mu(k)+\mu(l)$.
\item[(PMS3)] $\exists i\in I \text{ s.t. } \mu(i)>0$.
\item[(PMS4)] $\forall \alpha\in\mathbb{F}(\mathbb{N},I):$ If $\ell:=\lim_{m\rightarrow\infty}\mu(\bigwedge_{n=1}^m\alpha(n))$ exists and $\ell>0$, then there is a $x\in\bigcap_{n\in\mathbb{N}}\lambda_0^1(\alpha(n))$ (i.e. $\bigcap_{n\in\mathbb{N}}\lambda_0^1(\alpha(n))$ is inhabited).
\end{labeling}
\end{definition}

\begin{remark}\label{basic remarks premeasure space}
Let $i,j\in I$, then it follows that 
\begin{itemize}
\item $\bm\lambda_0(i\wedge j)\eqcompset \bm\lambda_0(i)\wedge\bm\lambda_0(j)$
\item $\bm\lambda_0(i\vee j)\eqcompset \bm\lambda_0(i)\vee\bm\lambda_0(j)$
\item $\bm\lambda_0(i\sim j)\eqcompset \bm\lambda_0(i)-\bm\lambda_0(j)$
\end{itemize}
from (PMS1) and the fact $\bm\lambda$ is a subfamily of $\bm\nu$. Also it follows from (PMS2) that $\mu(i)=\mu(i\wedge j)+\mu(i\sim j)$ for all $i,j\in I$.
\end{remark}

At the end of this section we want to give a concrete example of a pre-measure space. We have alredy seen that if we equip any set $X$ with the inequality
\begin{align*}
x\neq_X y \;:\Leftrightarrow\; \exists f\in\mathbb{F}(X,\bm 2) \text{ s.t. } f(x)\neq f(y)
\end{align*}
We can define the $\mathbb{F}(X,\bm 2)$-set of detachable  subsets, which classically corresponds to the powerset of $X$. The perhaps simplest example of a measure space in classical measure theory is the $\sigma$-algebra of all subsets together with the Dirac-measure, i.e. the measure concentrated at a single point. We want to give an analogous example for pre-measure spaces.\footnote{This example was first described \cite{Petrakis2019b}, but appears here in a slightly different form since we use a different notion of pre-measure space.}

Let $X$ be inhabited with $x_0\in X$ and equiped with the inequality $\neq_X$ defined above. From now on we regard $\bm 2$ as  asubset of $\mathbb{R}$. Define
\begin{align*}
\mu_{x_0}:\mathbb{F}(X,\bm 2)&\rightarrow\mathbb{R}_{\geq 0} \\
f&\mapsto f(x_0)
\end{align*}
We have to introduce operations on $\mathbb{F}(X,\bm 2)$, so let $f,g\in\mathbb{F}(X,\bm 2)$ and define
\begin{itemize}
\item $f\wedge g:=fg$
\item $f\vee g:=f+g-fg$
\item $\sim f:= 1-f$
\end{itemize}

Then $(X,\mathbb{F}(X,\bm 2),\bm\delta,\mathbb{F}(X,\bm 2),\bm\delta,\mu_{x_0})$ is a pre-measure space, where the embedding $\id_{\mathbb{F}(X,\bm 2)}:\mathbb{F}(X,\bm 2)\rightarrow\mathbb{F}(X,\bm 2)$ and the identity family maps make $\bm\delta$ into a subfamily of itself and the operations are the ones defined above for both index-sets. 

It remains to verify (PMS1)-(PMS4). Recall that for any $f\in\mathbb{F}(X,\bm 2)$ we have that $\chi_{\bm\delta_0(f)}\eqpf f$ and that complemented subsets are equal if and only if their characteristic functions are equal as partial functions, which gives us
\begin{itemize}
\item $\bm\delta_0(f\wedge g)\eqcompset\bm\delta_0(f)\wedge\bm\delta_0(g)$
\item $\bm\delta_0(f\vee g)\eqcompset\bm\delta_0(f)\vee\bm\delta_0(g)$
\item $\bm\delta_0(\sim f)\eqcompset-\bm\delta_0(f)$
\end{itemize}
Moreover, we have that
\begin{align*}
\mu_{x_0}(f)+\mu_{x_0}(g)=f(x_0)+g(x_0)-f(x_0)g(x_0)+f(x_0)g(x_0)=\mu_{x_0}(f\vee g)+\mu_{x_0}(f\wedge g)
\end{align*}
which finishes the verification of (PMS1) and we have that
\begin{align*}
\mu_{x_0}(f)=f(x_0)\big(1-g(x_0)\big)+f(x_0)g(x_0)=\mu_{x_0}(f\sim g)+\mu_{x_0}(f\wedge g)
\end{align*}
which finishes the verification of (PMS2), since the first part of (PMS2) holds trivially. To check (PMS3), observe that for $\bar{1}\in\mathbb{F}(X,\bm 2)$, the constant function $1$, we have that $\mu_{x_0}(\bar{1})=1$. Finally let $\alpha\in\mathbb{F}(\mathbb{N},\mathbb{F}(X,\bm 2))$ be such that
\begin{align*}
\ell:=\lim_{m\rightarrow\infty}\mu_{x_0}\bigg(\bigwedge_{n=1}^m\alpha_n\bigg)=\lim_{m\rightarrow\infty}\prod_{n=1}^m\alpha_n(x_0)
\end{align*}
exists and $\ell>0$. Then $\ell=1$ and in particular $\alpha_n(x_0)=1$ for all $n\in\mathbb{N}$, which means $x_0\in\bigcap_{n\in\mathbb{N}}\delta_0^1(\alpha_n)$. This proves (PMS4).

\section{The pre-integration space of simple functions}
In this section we start with a pre-measure space $(X,I,\bm\lambda,J,\bm\nu,\mu)$ and want to construct the corresponding pre-integration space of simple functions. Together with the previous example of the pre-measure space of detachable subsets with the Dirac-measure this gives us a first example of a pre-integration space. We start by proving a few basic facts about pre-measure spaces (see \cite[p. 283-284]{BB85}).
\begin{lemma}\label{pre-10.2}
Let $i\in I$ s.t. $\lambda_0^1(i)=\emptyset$, then $\mu(i)=0$
\end{lemma}
\begin{proof}
Assume that $\mu(i)>0$ and let $\alpha\in\mathbb{F}(\mathbb{N},I)$ be given by $\alpha(n):= i$ for all $n\in\mathbb{N}$. Then $\mu(i)=\lim_{m\rightarrow\infty}\mu\big(\bigwedge_{n=1}^m\alpha(n)\big)>0$ (in particular this limit exists). By (PMS4) we get that $\bigcap_{n\in\mathbb{N}}\lambda_0^1(\alpha(n))\eqsubset\lambda_0^1(i)$ is inhabited, which is a contradiction. Hence $\mu(i)=0$ by Lemma (2.18) of \cite{BB85} and the fact that $\mu(i)\geq 0$.
\end{proof}

\begin{lemma}\label{pre-10.3}
Let $i_1,...,i_n\in I$ and define $F:=\bigcap_{k=1}^n\big(\lambda_0^1(i_k)\cup\lambda_0^0(i_k)\big)$, then there is a $j\in I$ s.t. $\bm\lambda_0(j)=(\emptyset,F)$.
\end{lemma}
\begin{proof}
For each  $k=1,...,n$ we have that $\bm\lambda_0(i_k\sim i_k)=(\emptyset,\lambda_0^1(i_k)\cup\lambda_0^0(i_k))$. It follows that $\bm\lambda_0(\bigvee_{k=1}^n i_k\sim i_k)=\bigvee_{k=1}^n \bm\lambda_0(i_k\sim i_k)=(\emptyset,F)$.
\end{proof}

\begin{lemma}\label{pre-10.4}
Forall $j,i_1,...,i_n\in I$ and define $F:=\bigcap_{k=1}^n(\lambda_0^1(i_k)\cup\lambda_0^0(i_k))$ there is a $k\in I$ s.t. $\bm\lambda_0(k)=(\lambda_0^1(j)\cap F,\lambda_0^0(j)\cap F) $ and $\mu(k)=\mu(j)$.
\end{lemma}
\begin{proof}
Let $l:=\bigvee_{k=1}^n i_k\sim i_k\in I$ and define $k:=j\sim l\in I$. Note that $\bm\lambda_0(k)=(\lambda_0^1(j)\cap F,\lambda_0^0(j)\cap F)$ by lemma (\ref{pre-10.3}). Since $\bm\lambda_0(j\wedge l)=(\emptyset,(\lambda_0^1(j)\cup\lambda_0^0(j))\cap F)$, we get that $\mu(j\wedge l)=0$ by lemma (\ref{pre-10.2}). Thus by (PMS2) we get that
\begin{align*}
\mu(j)=\mu(j\wedge l)+\mu(k)=\mu(k)
\end{align*}
\end{proof}

\begin{lemma}\label{pre-10.5}
Let $i,j,i_1,...,i_n\in I$ and define $F:=\bigcap_{k=1}^n(\lambda_0^1(i_k)\cup\lambda_0^0(i_k))$, $F_i:=\lambda_0^1(i)\cup\lambda_0^0(i)$, $F_j:=\lambda_0^1(j)\cup\lambda_0^0(j)$ and $F':=F_i\cap F_j\cap F$. If $\chi_{\bm\lambda_0(i)}(x)\leq \chi_{\bm\lambda_0(j)}(x)$ for every $x\in F'$, then $\mu(i)\leq \mu(j)$.
\end{lemma}
\begin{proof}
Let $i',j'\in I$ be s.t.
\begin{itemize}
\item $\bm\lambda_0(i')=(\lambda_0^1(i)\cap F',\lambda_0^0(i)\cap F')$
\item $\bm\lambda_0(j')=(\lambda_0^1(j)\cap F',\lambda_0^0(j)\cap F')$
\end{itemize}
They exist by lemma (\ref{pre-10.4}) and we have $\mu(i)=\mu(i')$ and $\mu(j)=\mu(j')$. Since we have $\chi_{\bm\lambda_0(i)}\wedge \chi_{\bm\lambda_0(j)}=\chi_{\bm\lambda_0(i)}$ on $F'$, it follows that $j'\wedge i'=_I i'$. Let $k:=j'\sim i'$, then by (PMS2) we get that
\begin{align*}
\mu(i)=\mu(i')=\mu(j')-\mu(k)\leq \mu(j')=\mu(j)
\end{align*}
\end{proof}

We now want to introduce the set of simple functions over our pre-measure space. To do this we first have to define an appropriate index set, construct the family of simple functions over this index set and show that this family is indeed a set of partial functions.

\begin{definition}\label{simplefunctions}
Let $X$ be a set with an apartness-relation $\neq_X$, $I$ be a set and $\bm{\lambda}=(\lambda_0^1,\lambda_1^1,\mathcal{E}^1,\lambda_0^0,\lambda_1^0,\mathcal{E}^0)$ an $I$-family of complemented subsets of $X$. The set $S(I)$ is the totality $\sum_{n\in\mathbb{N}}(\mathbb{R}\times I)^n$ with the equality:
\begin{align*}
(a_k,i_k)_{k=1}^n =_{S(I)} (b_\ell,j_\ell)_{\ell=1}^m \;:\Leftrightarrow\;\sum_{k=1}^n a_k\cdot\chi_{\bm{\lambda}_0(i_k)}\eqpf\sum_{\ell=1}^m b_\ell\cdot\chi_{\bm{\lambda}_0(j_\ell)}
\end{align*}
We define the $S(I)$-family of (real-valued) partial functions $\bm\Lambda=(\lambda_0,\lambda_1,\mathcal{E},F)$ as follows:
\begin{enumerate}
\item $\lambda_0:S(I)\rightsquigarrow\vnull$ is the the assignment-routine given by
	\begin{align*}
	(a_k,i_k)_{k=1}^n \;\mapsto\; \bigcap_{k=1}^n\big(\lambda_0^1(i_k)\cup\lambda_0^0(i_k)\big)
	\end{align*}
\item $\mathcal{E}:\bigcurlywedge_{v\in S(I)}\mathbb{F}(\lambda_0(v),X)$ is the dependent assignment routine that associates each $v:=(a_k,i_k)_{k=1}^n$ with the embedding $\varepsilon_v:\bigcap_{k=1}^n\big(\lambda_0^1(i_k)\cup\lambda_0^0(i_k)\big)\hookrightarrow X$ induced by the embeddings $\varepsilon_{i_k}^1$ and $\varepsilon_{i_k}^0$
\item $\lambda_1:\bigcurlywedge_{(v,w)\in S(I)}\mathbb{F}(\lambda_0(v),\lambda_0(w))$ is the dependent assignment routine that associates each $(v,w)\in S(I)$ with the function $\lambda_{vw}:\lambda_0(v)\rightarrow\lambda_0(w)$ s.t.
	\begin{itemize}
	\item $\lambda_{vv}:=\id_{\lambda_0(v)}$ and
	\item $(\lambda_{vw},\lambda_{wv})$ witnesses the equality $v=_{S(I)}w$.\footnote{Here we use lemma (\ref{uniquewitness}), i.e. the fact that two functions witnessing an equality of subsets are equal.}
	\end{itemize}
\item $F:\bigcurlywedge_{v\in S(I)}\mathbb{F}(\lambda_0(v),\mathbb{R})$ is the dependent assignment routine that assigns each $v:=(a_k,i_k)_{k=1}^n$ the function $f_v:=\sum_{k=1}^n a_k\cdot\chi_{\bm{\lambda}_0(i_k)}$.
\end{enumerate}
\end{definition}

Note that from the definition of the equality $=_{S(I)}$ it immediatelly follows that $\bm\Lambda$ is an $S(I)$-set of partial functions. Also note that this construction can be done for any $I$-family $\bm\lambda$ of complemented subsets. We now return to the case where we have a pre-measure space $(X,I,\bm\lambda,J,\bm\nu,\mu)$ and want to show that we can define an integral on $S(I)$ that allows us to construct the pre-integration of simple functions. We first have to prove a few lemmas that use the additional structure on the index set $I$ provided by the pre-measure space.

\begin{lemma}\label{pre-7.8}
The assignment routine
\begin{align*}
\mathsf{disjrep}:S(I)&\rightsquigarrow S(I) \\
(a_k,i_k)_{k=1}^n &\mapsto \Big(\sum_{f(k)=1}a_k,\big(\bigwedge_{f(k)=1}i_k\big)\sim\big(\bigvee_{f(k)=0}i_k\big)\Big)_{f:\{1,...,n\}\rightarrow\bm 2}
\end{align*}
is a function and as such identical to $\id_{S(I)}$, i.e. for all $v\in S(I)$ we have $v=_{S(I)}\mathsf{disjrep}(v)$. If for $v\in S(I)$ we have that $\mathsf{disjrep}(v):=(b_\ell,j_\ell)_{\ell=1}^m$ then the $\bm\lambda_0(j_\ell)$ are \emph{disjoint}, i.e. for distinct $\ell$ and $k$ we have $\chi_{\bm\lambda_0(j_k)}\cdot\chi_{\bm\lambda_0(j_\ell)}=0$ on $\big(\lambda_0^1(j_k)\cup\lambda_0^0(j_k)\big)\cap\big(\lambda_0^1(j_\ell)\cup\lambda_0^0(j_\ell)\big)$. For $v\in S(I)$ we call $\mathsf{disjrep}(v)$ the \emph{disjoint representation} of $v$.
\end{lemma}
\begin{proof}
Let $v:=(a_k,i_k)_{k=1}^n$. For a function $f:\{1,...,n\}\rightarrow\bm 2$ define $j_f\in I$ by
\begin{align*}
j_f:=\Big(\bigwedge_{f(k)=1}i_k\Big)\sim\Big(\bigvee_{f(k)=0}i_k\Big)
\end{align*}
Then
\begin{align*}
\bm\lambda_0(j_f)=\Big(\bigwedge_{f(k)=1}\bm\lambda_0(i_k)\Big)\wedge\Big(\bigwedge_{f(k)=0}-\bm\lambda_0(i_k)\Big)
\end{align*}
Using this it is easy to verify that we have
\begin{align*}
\sum_{k=1}^n a_k\cdot\chi_{\bm\lambda_0(i_k)}\eqpf\sum_{f:\{1,...,n\}\rightarrow\bm 2}\bigg(\sum_{f(k)=1}a_k\bigg)\cdot\chi_{\bm\lambda_0(j_f)}
\end{align*}
Now let $f,g:\{1,...,n\}\rightarrow\bm 2$ be distinct functions, i.e. there is a $k\in\{1,...,n\}$ s.t. $f(k)=1$ but $g(k)=0$ or vice versa. For such a $k$ we get that $\lambda_0^1(j_f)\subseteq\lambda_0^1(i_k)$ and $\lambda_0^1(j_g)\subseteq\lambda_0^0(i_k)$ or vice versa. This implies that $\bm\lambda_0(j_f)$ and $\bm\lambda_0(j_g)$ are disjoint.
\end{proof}

\begin{lemma}\label{pre-7.8.2}
For each $v:=(a_k,i_k)_{k=1}^n\in S(I)$ we have
\begin{align*}
\sum_{k=1}^n a_k\cdot\mu(i_k)=\sum_{f:\{1,...,n\}\rightarrow\bm 2}\bigg(\sum_{f(k)=1}a_k\bigg)\cdot\mu(j_f)
\end{align*}
where $j_f\in I$ is defined for $f:\{1,...,n\}\rightarrow\bm 2$ as in the proof of the previous lemma.
\end{lemma}
\begin{proof}
Let $F:=\bigcap_{k=1}^n(\lambda_0^1(i_k)\cup\lambda_0^0(i_k))$, without loss of generality we can assume that for every $k\in\{1,...,n\}$ we have $\lambda_0^1(i_k)\cup\lambda_0^0(i_k)\subseteq F$ since otherwise we can consider $i_k'\in I$ s.t. $\bm\lambda_0(i_k')=(\lambda_0^1(i_k)\cap F,\lambda_0^0(i_k)\cap F)$. By lemma (\ref{pre-10.4}) we have that $\mu(i_k)=\mu(i_k')$ for each $k\in\{1,...,n\}$ and since for each $f:\{1,...,n\}\rightarrow\bm 2$ the domain of $\chi_{\bm\lambda_0(j_f)}$ is $F$ (with $j_f\in I$ as in the proof of lemma (\ref{pre-7.8})), we have that
\begin{align*}
\bm\lambda_0(j_S)=\Big(\bigwedge_{f(k)=1}\bm\lambda_0(i_k)\Big)\wedge\Big(\bigwedge_{f(k)=0}-\bm\lambda_0(i_k)\Big)=\Big(\bigwedge_{f(k)=1}\bm\lambda_0(i_k')\Big)\wedge\Big(\bigwedge_{f(k)=0}-\bm\lambda_0(i_k')\Big)
\end{align*}
For $k\in\{1,...,n\}$, let $P(k)$ be the set of functions $f:\{1,...,n\}\rightarrow\bm 2$ s.t. $f(k)=1$. Note that for each $x\in F$ we have
\begin{align*}
x\in\lambda_0^1(i_k) \; &\Leftrightarrow\; x\in\bigcup_{f\in P(k)}\lambda_0^1(j_f) \text{ and}\\
x\in\lambda_0^0(i_k) \; &\Leftrightarrow\; x\in\bigcap_{f\in P(k)}\lambda_0^0(j_f)
\end{align*}
It follows that $\chi_{\bm\lambda_0(i_k)}=\bigvee_{f\in P(k)}\chi_{\bm\lambda_0(j_f)}$ on $F$ and since $\lambda_0^1(i_k)\cup\lambda_0^0(i_k)\subseteq F$ we get that $\bm\lambda_0(i_k)=\bigvee_{f\in P(k)}\bm\lambda_0(j_f)$. 

Furthermore, note that for any $i,j\in I$ s.t. $\bm\lambda_0(i)$ and $\bm\lambda_0(j)$ are disjoint we have that $\lambda_0^1(i)\cap\lambda_0^1(j)=\emptyset$ and thus by lemma (\ref{pre-10.2}) it follows that $\mu(i\wedge j)=0$. Hence by (PMS1) $\mu(i_k)=\sum_{f\in P(k)}\mu(j_f)$. Using this it is easy to see that
\begin{align*}
\sum_{k=1}^n a_k\cdot\mu(i_k)&=\sum_{k=1}^n a_k\cdot\Big(\sum_{f\in P(k)}\mu(j_f)\Big) \\[1em]
&=\sum_{f:\{1,...,n\}\rightarrow\bm 2}\Big(\sum_{k\text{ s.t. } f\in P(k)}a_k\Big)\cdot\mu(j_f) \\[1em]
&=\sum_{f:\{1,...,n\}\rightarrow\bm 2}\Big(\sum_{f(k)=1}a_k\Big)\cdot\mu(j_f)
\end{align*}
\end{proof}

\begin{lemma}\label{pre-10.6}
Let $v:=(a_k,i_k)_{k=1}^n,w:=(b_\ell,j_\ell)_{\ell=1}^m\in S(I)$ be s.t. 
\begin{align*}
\sum_{k=1}^n a_k\cdot\chi_{\bm\lambda_0(i_k)}(x)\leq\sum_{\ell=1}^m b_\ell\cdot\chi_{\bm\lambda_0(j_\ell)}(x)
\end{align*}
for all $x\in F:=\Big(\bigcap_{k=1}^n(\lambda_0^1(i_k)\cup\lambda_0^0(i_k))\Big)\cap\Big(\bigcap_{\ell=1}^m(\lambda_0^1(j_\ell)\cup\lambda_0^0(j_\ell))\Big)$, then
\begin{align*}
a:=\sum_{k=1}^n a_k\cdot\mu(i_k)\leq b:=\sum_{\ell=1}^m b_\ell\cdot\mu(j_\ell)
\end{align*}
\end{lemma}
\begin{proof}
Without loss of generality we can assume that $n=m$ and for each $k\in\{1,...,n\}$ we have $i_k=_I j_k$, because otherwise we can add appropriate indices with coefficient $0$ to both functions respectively. Now, assume that $a>b$, i.e. that $\sum_{k=1}^n(a_k-b_k)\cdot\mu(\bm\lambda_0(i_k))>0$. By lemma (\ref{pre-7.8.2}) we get that
\begin{align*}
\sum_{f:\{1,...,n\}\rightarrow\bm 2}\Big(\sum_{f(k)=1}a_k-\sum_{f(k)=1}b_k\Big)\cdot\mu(j_f)>0
\end{align*}
By (2.16) of \cite{BB85} there is a function $f:\{1,...,n\}\rightarrow\bm 2$ s.t. $\Big(\sum_{f(k)=1}a_k-\sum_{f(k)=1}b_k\Big)>0$ and $\mu(j_f)>0$. From the proof of lemma (\ref{pre-10.2}) it follows that there is an $x\in\lambda_0^1(j_f)\subseteq F$. For each $k\in\{1,...,n\}$ we get that
\begin{align*}
\chi_{\bm\lambda_0(i_k)}(x)=1 \;\Leftrightarrow\; f(k)=1
\end{align*}
and thus
\begin{align*}
\sum_{f(k)=1}a_k=\sum_{k=1}^n a_k\cdot\chi_{\bm\lambda_0(i_k)}(x)\leq\sum_{k=1}^n b_k\cdot\chi_{\bm\lambda_0(i_k)}(x)=\sum_{f(k)=1}b_k
\end{align*}
which is a contradiction. Hence $a\leq b$.
\end{proof}
We can now define the integral on the simple functions and show that it is indeed a function.

\begin{lemma}\label{pre-10.7}
The assignment-routine
\begin{align*}
\int\_ d\mu:S(I)&\rightsquigarrow\mathbb{R} \\
(a_k,i_k)_{k=1}^n &\mapsto \sum_{k=1}^n a_k\cdot\mu(i_k)
\end{align*}
is a function.
\end{lemma}
\begin{proof}
Let $v:=(a_k,i_k)_{k=1}^n,w:=(b_\ell,j_\ell)_{\ell=1}^m\in S(I)$ be s.t. $v=_{S(I)}w$ (i.e. $f_v\eqpf f_w$) and let $F:=\lambda_0(v)\cap\lambda_0(w)$. We thus have $f_v\leq f_w$ on $F$ and $f_v\geq f_w$ on $F$, and thus by lemma (\ref{pre-10.6}) 
\begin{align*}
\Big(\int v \;d\mu\leq\int w \;d\mu \;\;\&\;\int v \;d\mu\geq\int w \;d\mu\Big) \;\Rightarrow\;\int v\; d\mu=\int w \;d\mu
\end{align*}
\end{proof}

It remains to show that $(X,S(I),\bm\Lambda,\int\_d\mu)$ is a pre-integration space. For this we need two more lemmas. Since we don't want to use countable choice, the next lemma and its proof is a bit more involved than lemma 10.8 in \cite[p. 284]{BB85}.

\begin{lemma}\label{pre-10.8}
Let $S^+(I):=\{v\in S(I): f_v\geq0\}$, then there is a function ${\phi:\mathbb{N}\rightarrow S^+(I)\rightarrow I}$ with $N\mapsto (\phi_N:S^+(I)\rightarrow I)$, s.t. for all $N\in\mathbb{N}$ and $v\in S(I)$ we have
\begin{itemize}
\item $\lambda_0^1(\phi_N(v))\cup\lambda_0^0(\phi_N(v))\subseteq\lambda_0(v)$
\item $\forall x\in \lambda_0^0(\phi_N(v)):\; f_v(x)<N^{-1}$
\item $\mu(\phi_N(v))\leq 2N\int v\;d\mu$
\end{itemize}
\end{lemma}
\begin{proof}
Fix $N\in\mathbb{N}$ and let $v\in S^+(I)$ and $\mathsf{disjrep}(v):=(a_k,i_k)_{k=1}^n$. Then $a_k\geq0$ for all $k=1,...,n$. Let $f,g:\{1,...,n\}\rightarrow\bm 2$ be functions s.t.
\begin{enumerate}
\item $\forall k\in\{1,...,n\}:\;f(k)\vee g(k)=1$
\item $f(k)=1 \;\Rightarrow\; a_k<N^{-1}$
\item $g(k)=1 \;\Rightarrow\; a_k>(2N)^{-1}$
\end{enumerate}
Such $f$ and $g$ exist by corollary (2.17) of \cite{BB85}. We have that
\begin{align*}
\Big(\forall k\in\{1,...,n\}:\;f(k)=1\Big)\quad\text{or}\quad\Big(\exists k\in\{1,...,n\}\text{ s.t. }g(k)=1\Big)
\end{align*}
Now assume we have another pair of functions $f',g':\{1,...,n\}\rightarrow\bm 2$ s.t. the conditions (i)-(iii) hold. Let $f'':=f\vee f'$ and $g'':=g\vee g'$. It is easy to verify that $f'',g''$ also fulfill (i)-(iii). 

Since there are only finitely many pairs of functions $\{1,...,n\}\rightarrow\bm 2$ s.t. (i)-(iii) hold and they are closed under taking (pointwise) maxima we can consider the \emph{maximal} functions $\varphi_v,\psi_v:\{1,...,n\}\rightarrow\bm 2$ s.t. (i)-(iii) hold. In particular they will be given by
\begin{align*}
\Bigg(\varphi_v=\bigvee_{\substack{\exists g:\{1,...,n\}\rightarrow\bm 2\\ (f,g)\text{ fulfill (i)-(iii)}}}f\Bigg) \quad\quad \&\quad\quad
\Bigg(\psi_v=\bigvee_{\substack{\exists f:\{1,...,n\}\rightarrow\bm 2\\ (f,g)\text{ fulfill (i)-(iii)}}}g\Bigg)
\end{align*}
Now, let
\begin{align*}
\phi_N(v):=\begin{cases}
\bigvee_{k=1}^n(i_k\sim i_k), \text{ if } \forall k\in\{1,...,n\}:\;\varphi_v(k)=1 \\
\Big(\bigvee_{\psi_v(k)=1}i_k\Big)\sim\Big(\bigvee_{k=1}^n(i_k\sim i_k)\Big),\text{ else}
\end{cases}
\end{align*}
Notice that this is a valid case distinction. In order to verify that $\phi_N(v)$ does indeed satisfy the three conditions of the lemma, observe that $\bm\lambda_0(\bigvee_{k=1}^n(i_k\sim i_k))=(\emptyset,\lambda_0(v))$ and for $i_v:=\bigvee_{\psi_v(k)=1}i_k$ we have
\begin{align*}
\bm\lambda_0\big(i_v\sim\bigvee_{k=1}^n(i_k\sim i_k)\big)=\big(\lambda_0^1(i_v)\cap\lambda_0(v),\lambda_0^0(i_v)\cap\lambda_0(v)\big)
\end{align*}
First assume that $\Big(\forall k\in\{1,...,n\}:\;\varphi_v(k)=1\Big)$. Then for $j:= \bigvee_{k=1}^n(i_k\sim i_k)$ we have $\phi_N(v)=_I j$ and clearly $\lambda_0^1(j)\cup\lambda_0^0(j)\subseteq\lambda_0(v)$ and $\big(\forall x\in\lambda_0^0(j):\; f_v(x)<N^{-1}\big)$. Note that $(2N)^{-1}\cdot\chi_{\bm\lambda_0(j)}\leq f_v$ on $\lambda_0(v)$. From this, the third property follows by lemma (\ref{pre-10.6}).

If $\Big(\exists k\in\{1,...,n\}:\;\varphi_v(k)=0\Big)$, i.e. $\psi_v(k)=1$, let $j:=i_v\sim\bigvee_{k=1}^n(i_k\sim i_k)$. Then we have $\phi_N(v)=_I j$ and clearly $\lambda_0^1(j)\cup\lambda_0^0(j)\subseteq\lambda_0(v)$. Let $x\in\lambda_0^0(j)$. Then $x\in\bigcap_{\psi_v(k)=1}\lambda_0^0(i_k)$ and thus $f_v(x)=0$ or there is a $k\in \{1,...,n\}$ with $\varphi_v(k)=1$ s.t. $x\in\lambda_0^1(i_k)$, in which case $f_v(x)=a_k<N^{-1}$ since the $\bm\lambda_0(i_k)$ are disjoint. In either case the second condition is fulfilled.

Using lemma (\ref{pre-10.4}) and the fact that the $\bm\lambda_0(i_k)$ are disjoint we get that
\begin{align*}
\mu(j)=\mu(i_v)=\sum_{\psi_v(k)=1}\mu(i_k)
\leq\sum_{\psi_v(k)=1}(2N)^{-1}\cdot a_k\cdot\mu(i_k)\leq (2N)^{-1}\int v\;d\mu
\end{align*}
Here the last inequality follows from lemma (\ref{pre-10.6}) using the fact that on $\lambda_0(v)$ we have $\sum_{\psi_v(k)=1}a_k\cdot\chi_{\bm\lambda_0(i_k)}\leq f_v$. 

It remains to check that $\phi_N$ is a function. So let $w\in S^+(I)$ s.t. $v=_{S(I)}w$ (i.e. $f_v\eqpf f_w$) and $\mathsf{disjrep}(w):=(b_\ell,j_\ell)_{\ell=1}^m$. Since $\bm\lambda$ is an $I$-set of complemented subsets it suffices to check that $\bm\lambda_0(\phi_N(v))\eqcompset\bm\lambda_0(\phi_N(w))$. 

Note that by the above remarks $\chi_{\bm\lambda_0(\phi_N(v))}$ and $\chi_{\bm\lambda_0(\phi_N(w))}$ have the same domain $\lambda_0(v)\eqsubset\lambda_0(w)$. Now, assume that for $x\in\lambda_0(v)$ we have $\chi_{\bm\lambda_0(\phi_N(v))}(x)=1$, i.e. $\lambda_0^1(\phi_N(v))$ is inhabited and we have $x\in\lambda_0^1(i_k)$ for some $k\in\{1,...,n\}$ with $\psi_v(k)=1$. Thus $f_v(x)=a_k>(2N)^{-1}>0$ and since $f_v=f_w$ we must also have $x\in\lambda_0^1(j_\ell)$ for some $\ell=1,...,m$. Moreover, $b_\ell=a_k>(2N)^{-1}$ and hence $\psi_w(\ell)=1$ by the maximality of $\psi_w$. It follows that $x\in\lambda_0^1(i_w)$ where $i_w:=\bigvee_{\psi_w(\ell)=1}j_\ell$.

Also note that since $\lambda_0^1(\phi_N(v))$ is inhabited there is a $k=1,...,n$ s.t. $\varphi_v(k)=0$ and hence we get that $a_k\geq N^{-1}$ and $\lambda_0^1(i_k)$ is inhabited by the maximality of $\varphi_v$. Thus, since $f_v=f_w$, there is an $\ell=1,...,m$ s.t. $b_\ell\geq N^{-1}$ and $\lambda_0^1(j_\ell)$ is inhabited, i.e. there is a $\ell=1,...,m$ s.t. $\varphi_w(\ell)=0$ and  $\phi_N(w)=i_w\sim\bigvee_{\ell=1}^m(j_\ell\sim j_\ell)$. It follows that $\chi_{\bm\lambda_0(\phi_N(w))}(x)=1$. Repeating the argument for $x\in\lambda_0(w)$ with  $\chi_{\bm\lambda_0(\phi_N(w))}(x)=1$ we can conclude that
\begin{align*}
\Big(\lambda_0^1(\phi_N(v))\cup\lambda_0^0(\phi_N(v)\Big)=\lambda_0(v)&=\lambda_0(w)=\Big(\lambda_0^1(\phi_N(w))\cup\lambda_0^0(\phi_N(w)\Big) \\
\&\quad\quad\forall x\in\lambda_0(v):\; \chi_{\bm\lambda_0(\phi_N(v))}(x)=1\;&\Leftrightarrow\;\chi_{\bm\lambda_0(\phi_N(w))}(x)=1
\end{align*}
which finishes the proof.
\end{proof}

\begin{lemma}\label{pre-10.9}
Let $v:=(a_k,i_k)_{k=1}^n\in S(I)$ and $c>0$ be s.t. $f_v\leq c$ on $\lambda_0(v)$. Let $i\in I$ be s.t. $f_v\leq 0$ on $\lambda_0^0(i)\cap\lambda_0(v)$, then for every $\varepsilon>0$ there exists a $j\in I$ s.t.
\begin{enumerate}
\item $\lambda_0^1(j)\cup\lambda_0^0(j)\subseteq\lambda_0(v)$
\item $\forall x\in\lambda_1^0(j):\; f_v(x)>\varepsilon$
\item $\mu(j)\geq c^{-1}\big(\int v\;d\mu - 2\varepsilon\mu(i)\big)$
\end{enumerate}
\end{lemma}
\begin{proof}
Again w.l.o.g. we can assume that the $\bm\lambda_0(i_k)$ are disjoint (otherwise replace $v$ by $\mathsf{disjrep}(v)$). Let $f,g:\{1,...,n\}\rightarrow\bm 2$ be functions s.t.
\begin{enumerate}
\item $\forall k\in\{1,...,n\}:\;\text{either } f(k)=1 \text{ or } g(k)=1$
\item $f(k)=1 \;\Rightarrow\; a_k<N^{-1}$
\item $g(k)=1 \;\Rightarrow\; a_k>(2N)^{-1}$
\end{enumerate}
Such $f$ and $g$ exist by corollary (2.17) of \cite{BB85}. We have that
\begin{align*}
\Big(\forall k\in\{1,...,n\}:\;f(k)=1\Big)\quad\text{or}\quad\Big(\exists k\in\{1,...,n\}:\;g(k)=1\Big)
\end{align*}
In the first case, let $j\in I$ be such that $\bm\lambda_0(j)=(\emptyset,\lambda_0(v))$ (such a $j$ exists by lemma (\ref{pre-10.3})). Then clearly $\lambda_0^1(j)\cup\lambda_0^0(j)\subseteq\lambda_0(v)$ and $\forall x\in\lambda_1^0(j):\; f_v(x)>\varepsilon$ is vacuously true. By lemma (\ref{pre-10.2}) $\mu(j)=0$, so in order to check (iii) we need to prove 
\begin{align*}
2\varepsilon\mu(i)\geq \int v\;d\mu
\end{align*}
which follows from lemma (\ref{pre-10.6}) and the fact that $f_v\leq 2\varepsilon\chi_{\bm\lambda_0(i)}$ on $\lambda_0(v)\cap\big(\lambda_0^1(i)\cup\lambda_0^0(i)\big)$. This inequality holds since $a_k<2\varepsilon$ for all $k$ and $f_v\leq 0$ on $\lambda_0^0(i)\cap\lambda_0(v)$.

Now assume that $\Big(\exists k\in\{1,...,n\}:\;g(k)=1\Big)$ is inhabited. Let $j':=\bigvee_{g(k)=1}i_k$ and $j\in I$ s.t. $\bm\lambda_0(j)=\big(\lambda_0^1(j')\cap\lambda_0(v),\lambda_0^0(j')\cap\lambda_0(v)\big)$. Then clearly $\lambda_0^1(j)\cup\lambda_0^0(j)\subseteq\lambda_0(v)$. If $x\in\lambda_0^1(j)$ then in particular $x\in\lambda_0^1(i_k)$ for some $k\in \{1,...,n\}$ with $g(k)=1$ and hence $f_v(x)=a_k>\varepsilon$.

For each $k\in\{1,...,n\}$ let $j_k:=i_k\wedge i$ and define $w:=(a_k,i_k\wedge i)_{k=1}^n\in S(I)$. Using lemma (\ref{pre-10.6}) we get that
\begin{align*}
\int v\;d\mu &\leq\int w\;d\mu \\
&=\sum_{f(k)=1}a_k\mu(j_k) \;+\;\sum_{g(k)=1}a_k\mu(j_k) \\
&\leq 2\varepsilon\Big(\sum_{f(k)=1}\mu(j_k)\Big) \;+\;c\Big(\sum_{g(k)=1}\mu(j_k)\Big) \\
&\leq 2\varepsilon\mu(i)+c\mu(j)
\end{align*}
Note that the $\bm\lambda_0(j_k)$ are disjoint, which allows us to make the following estimates for the last inequality:
\begin{align*}
\sum_{f(k)=1}\mu(j_k)=\mu(\bigvee_{f(k)=1}j_k)=\mu(i\wedge\bigvee_{f(k)=1}i_k)\leq \mu(i)
\end{align*}
and 
\begin{align*}
\sum_{g(k)=1}\mu(j_k)\leq \sum_{g(k)=1}\mu(i_k)=\mu(j')=\mu(j)
\end{align*}
\end{proof}

\begin{theorem}\label{pre-10.10}
$(X,S(I),\bm\Lambda,\int\_d\mu)$ is a pre-integration space
\end{theorem}
\begin{proof}
We first have to define the required assignment routines, so let $v,w\in S(I)$ s.t. $\mathsf{disjrep}(v):=(a_k,i_k)_{k=1}^n$ and $\mathsf{disjrep}(w):=(b_\ell,j_\ell)_{\ell=1}^m$, $a\in\mathbb{R}$ and define:
\begin{align*}
a\cdot v&:=(a\cdot a_k,i_k)_{k=1}^n \\
v+w&:=\big((a_1,i_1),...,(a_n,i_n),(b_1,j_1),...,(b_m,j_m)\big) \\
\abs{v}&:=(\abs{a_k},i_k)_{k=1}^n \\
\wedge_1(v)&:=(a_k\wedge 1,i_k)_{k=1}^n
\end{align*}
It is easy to check that these assignment routines fulfill (PIS1). 

By (PMS3) there is an $i\in I$ s.t. $\mu(i)>0$. Let $v:=(\mu(i)^{-1},i)\in S(I)$, then $\int v\;d\mu=1$, which verifies (PIS3). Let $v\in S(I)$ with $\mathsf{disjrep}(v):=(a_k,i_k)_{k=1}^n$ and $\alpha,\beta\in\mathbb{F}(\mathbb{N}, S(I))$ be given by  $\alpha(m)=m\cdot (\wedge_1(m^{-1}\cdot v))$ and $\beta(m)=m^{-1}\cdot(\wedge_1(m\cdot\abs{v}))$ for $m\in\mathbb{N}$. Let $M_0\in\mathbb{N}$ be s.t. $M_0>a_1\vee...\vee a_n$ and $m\geq M_0$. Then $f_{\alpha(m)}=f_v$, i.e. $\alpha(m)=_{S(I)}v$ and thus $\lim_{n\rightarrow\infty}\int \alpha(n)\;d\mu=\int v\;d\mu$. Now, let $M_1\in\mathbb{N}$ s.t. $M_1^{-1}<\vert a_1\vert\wedge...\wedge\vert a_n\vert$ and $m\geq M_1$. Then 
\begin{align*}
\int \beta(m)\;d\mu=\frac{n}{m}\sum_{k=1}^n\mu(i_k)\xrightarrow{m\rightarrow\infty}0
\end{align*}

It remains to check (PIS2). So let $v\in S(I)$ with $v:=(a_k,i_k)_{k=1}^n$ and $\alpha\in\mathbb{F}(\mathbb{N},S(I))$ s.t. for all $n\in\mathbb{N}$ we have $f_{\alpha(n)}\geq 0$ and $\ell:=\sum_{k=1}^\infty\int\alpha(k)\;d\mu$ exists and $\ell<\int i\;d\mu$. Let $M:\mathbb{N}\rightarrow\mathbb{N}$ be the (strictly increasing) modulus of convergence with which $\big(\sum_{k=1}^n\int\alpha(k)\;d\mu\big)_{n=1}^\infty$ converges to $\ell\in\mathbb{R}$, i.e. for each $p\in\mathbb{N}$ we have that
\begin{align*}
\abs*{\ell-\sum_{k=1}^{M(p)}\int\alpha(k)\;d\mu}=\sum_{k=M(p)+1}^\infty\int\alpha(k)\;d\mu\leq\frac{1}{2^p}
\end{align*}
Let $i:=\bigvee_{k=1}^n i_k$ and note that $f_v=0$ on $\lambda_0^0(i)$. Furthermore take $c>0$ s.t. $f_v\leq c$ on $\lambda_0(v)$ and define
\begin{itemize}
\item $r:=\int v \;d\mu-\ell$
\item $\varepsilon:=\nicefrac{1}{2}\big(1+\mu(i)\big)^{-1}r$
\item $\alpha:=\nicefrac{1}{c}\big(r-2\varepsilon\mu(i)\big)$
\end{itemize}
Define the strictly increasing sequence $\eta:\mathbb{N}\rightarrow\mathbb{N}$ by $\eta(1):=1$ and for $n\geq 2$ by $\eta(n):=M(2n+1+p)-1$, where $p\in\mathbb{N}$ is s.t. $2^{-p}<\alpha\varepsilon$. Then for $n\geq 2$ we have that
\begin{align*}
\sum_{k=\eta(n)}^{\eta(n+1)-1}\int \alpha(k)\;d\mu \leq \frac{1}{2^{(2n+1)}}\alpha\varepsilon
\end{align*}
Let $\gamma\in\mathbb{F}(\mathbb{N},S(I))$ be given by $\gamma(n):=\sum_{k=\eta(n)}^{\eta(n+1)-1} \alpha(k)$ for all $n\in\mathbb{N}$.
Let $q\in\mathbb{N}$ be s.t. $q^{-1}<\varepsilon$ and for $k\geq 2$ let $j_k:=\phi_{2^kq}(\gamma(k))$ with $\phi$ as defined in lemma (\ref{pre-10.8}). It follows that $\lambda_0^1(j_k)\cup\lambda_0^0(j_k)\subseteq\lambda_0(\gamma(k))$ and for all $x\in\lambda_0^0(j_k)$ we have $f_{\gamma(k)}(x)<2^{-k}\varepsilon$ and
\begin{align*}
\mu(j_k)\leq \frac{2^{(k+1)}}{\varepsilon}\int \gamma(k)\;d\mu \leq 2^{-k}\alpha
\end{align*}
Furthermore, let $j\in I$ be s.t. for $w:=v-{\gamma(1)}$ we have that $\lambda_0^1(j)\cup\lambda_0^0(j)\subseteq\lambda_0(w)$ and for all $x\in\lambda_0^1(j)$ we have $f_w(x)>\varepsilon$ and
\begin{align*}
\mu(j)\geq c^{-1}\bigg(\int w\;d\mu - 2\varepsilon\mu(i)\bigg)\geq c^{-1}\big(r- 2\varepsilon\mu(i)\big)=\alpha
\end{align*}
Such a $j$ exists by lemma (\ref{pre-10.9}), since we have $0\leq f_{\gamma(1)}\leq\ell$ and thus $f_w\leq f_v<c$ on $\lambda_0(w)$ and $f_w\leq f_v=0$ on $\lambda_0^0(i)$. Finally, let $\delta\in\mathbb{F}(\mathbb{N},I)$ be given by $\delta(n):=\bigwedge_{k=2}^n(j-j_k)$ for $n\geq 2$. We now claim that $\ell':=\lim_{n\rightarrow\infty}\mu\big(\delta(n)\big)$ exists and that
\begin{align*}
\ell'\geq \mu(j)-\sum_{k=2}^\infty 2^{-k}\alpha\geq \frac{\alpha}{2}>0
\end{align*}
From this claim it follows by (PMS4) that there is an inhabitant
\begin{align*}
x\in\bigcap_{n\in\mathbb{N}}\lambda_0^1(\delta(n))=\lambda_0^1(j)\cap\big(\bigcap_{n\in\mathbb{N}}\lambda_0^0(j_n)\big)
\end{align*}
For $k\geq 2$ we get that $f_{\gamma(k)}(x)<2^{-k}\varepsilon$ and thus $\ell'':=\lim_{n\rightarrow\infty}f_{\beta(n)}(x)=\sum_{k=1}^\infty f_{\gamma(k)}(x)$ exists and 
\begin{align*}
f_v(x)-\ell''=\underbrace{f_v(x)-f_{\gamma(1)}(x)}_{=f_w(x)}-\sum_{k=2}^\infty f_{\gamma(k)}(x) \geq\varepsilon -\sum_{k=2}^\infty 2^{-k}\varepsilon=\frac{\varepsilon}{2}>0
\end{align*}
This completes the verification of (PIS2) so it remains to proof the claim. 

For $m>n\geq 2$ we have:
\begin{align*}
0 	&\leq \mu(\delta(n))-\mu(\delta(m)) \quad\text{by lemma (\ref{pre-10.5})}\\
	&\overset{(*)}{\leq}\mu\big(\delta(n)-\delta(m)\big) 	\\
	&=\mu\big( (j-\bigvee_{k=2}^nj_k)-(j-\bigvee_{k=2}^mj_k)\big) \\
	&\overset{(**)}{\leq}\mu\big( \bigvee_{k=2}^mj_k-\bigvee_{k=2}^nj_k\big)\\
	&\overset{(***)}{\leq}\mu\big(\bigvee_{k=n+1}^m j_k\big)\leq \sum_{k=n+1}^m\mu(j_k)\leq\sum_{k=n+1}^m 2^{-k}\alpha
\end{align*}
It follows that the sequence $\big(\mu(\delta(n))\big)_{n\in\mathbb{N}}$ is a Cauchy-sequence and hence its limit exist. Furthermore, we get
\begin{align*}
\lim_{n\rightarrow\infty}\mu(\delta(n))&=\lim_{n\rightarrow\infty}\mu\big(j-\bigvee_{k=2}^nj_k\big) \\
&\overset{(*)}{\geq} \mu(j)-\lim_{n\rightarrow\infty}\mu(\bigvee_{k=2}^n j_k)\\
&\geq \mu(j)-\sum_{k=2}^\infty 2^{-k}\alpha \geq \frac{\alpha}{2}>0
\end{align*}
Note that we used the following inequalities: Let $A,B,C\in I$, we have 
\begin{align*}
(*) \quad \mu(A-B)\overset{\mathsmaller{\text{(PMS2)}}}{=}\mu(A)-\mu(A\wedge B)\overset{\mathsmaller{\text{lemma (\ref{pre-10.5})}}}{\geq} \mu(A)-\mu(B) 
\end{align*}
\begin{align*}
(**) \quad \mu((A-B)-(A-C))&=\mu((A-B)\wedge(-A\vee C))\\
&=\mu\big(((A-B)-A)\vee (A\wedge(C-B))\big) \\
&\leq \underbrace{\mu((A-A)-B)}_{\mathsmaller{\leq\mu(A-A)=0\text{ by lemma (\ref{pre-10.2})}}} +\quad\mu(A\wedge(C-B)\leq\mu(C-B)
\end{align*}
\begin{align*}
(***) \quad \mu(A\vee B -B)=\mu((A-B)\vee(B-B))\leq \mu(A-B)+\mu(B-B)\leq\mu(A)
\end{align*}
\end{proof}
\newpage
This finishes the section on the simple functions. Note that most arguments have been directly translated from section 10 of chapter 6 of \cite{BB85} to our setting of pre-measure and pre-integration spaces. However, we were able to make certain things more precise,  like the formulation and proof of lemma (\ref{pre-10.8}) and the proof of theorem (\ref{pre-10.10}) and thus avoided using the axiom of countable choice.

\section{Complete extension of a pre-integration space}
Recall that we haven't addressed the most serious problems of BCMT from a predicative perspective. The construction of the complete extension of an integration space relies on the assumption that the totality of integrable functions forms a set, which is not acceptable from a predicative viewpoint. For the remainder of this paper we want to propose a solution to this problem. As we will see it suffices to consider the totality of canonically integrable functions, which does form a set, in order to construct the complete extension.

\subsection{The 1-norm of a pre-integration space}
Before we can give a predicative account of the complete extension of a pre-integration space, we have to make precise in what way this extended pre-integration space is the completion of the first one. In this section we will show how we can define a norm on the index set of the set of partial functions of a pre-integration space. We start by proving some basic results about pre-integration spaces (see \cite[pp. 217-218]{BB85}) that will be useful later.

\begin{lemma}\label{pre-1.2}
Let $i\in I$ and $\alpha\in\mathbb{F}(\mathbb{N},I)$ be such that for all $n\in\mathbb{N}$ we have $f_{\alpha(n)}\geq 0$ and $\sum_n\int\alpha(n)$ extists and $\int i+\sum_n\int\alpha(n)>0$. Then there exists $x\in\lzi\cap\bigcap_n\lambda(\alpha(n))$ such that $\sum_n f_{\alpha(n)}(x)$ exists and $f_i(x)+\sum_n f_{\alpha(n)}(x)>0$.
\end{lemma}
\begin{proof}
Let $N\in\mathbb{N}$ be s.t. $\sum_{n=N+1}^\infty\int\alpha(n)<\frac{1}{2}\big(\int i+\sum_n\int\alpha(n)\big)$, then using induction, we get that
\begin{align*}
\sum_{n=N+1}^\infty\int\alpha(n)<\int i+\sum_{n=1}^N\int\alpha(n)=\int\bigg(i+\sum_{n=1}^N\alpha(n)\bigg)
\end{align*}
From (PIS2) it follows that there is a $x\in\lzi\cap\bigcap_n\lambda(\alpha(n))$ such that $\sum_{n=N+1}^\infty f_{\alpha(n)}(x)$ exists and $f_i(x)+\sum_{n=1}^N f_{\alpha(n)}(x)>\sum_{n=N+1}^\infty f_{\alpha(n)}(x)$. Now the result follows since for all $n\in\mathbb{N}$ we have that $f_{\alpha(n)}\geq 0$.
\end{proof}

\begin{lemma}\label{pre-1.4}
$\forall i\in I:\;f_i\geq 0\;\Rightarrow\;\int i\geq 0$
\end{lemma}
\begin{proof}
Let $i\in I$ and assume that $f_i\geq 0$ but $\int i<0$. Consider the sequence $(0\cdot i,0\cdot i,...)$ then $f_{0\cdot i}\eqpf 0f_i=0$ and $\sum_n\int 0\cdot i=0$ and $\int (-i)+\sum_n\int 0\cdot i >0$. Hence, by lemma (\ref{pre-1.2}) there is an $x\in\lambda_0(i)$ s.t. $-f_i(x)+\sum_n0f_i(x)=-f_i(x)>0$ which is a contradiction.
\end{proof}

\begin{lemma}\label{pre-1.5}
$\forall i\in I:\;\abs*{\int i}\leq \int\abs{i} $
\end{lemma}
\begin{proof}
The result is immediatelly obtained by applying lemma (\ref{pre-1.4}) to $\abs{i}-i$ and $\abs{i}+i$ respectively.
\end{proof}

\begin{lemma}\label{pre-1.4.2}
Let $i,j\in I$ be s.t. for all $x\in\lzi\cap\lzj$ we have $f_i(x)\leq f_j(x)$, then $\int i\leq\int j$
\end{lemma}
\begin{proof}
Consider $j-i\in I$, then $\lambda_0(j-i)\eqsubset\lzi\cap\lzj$ and thus $f_{j-i}\geq 0$. Hence, by lemma (\ref{pre-1.4})
\begin{align*}
\int j-\int i=\int (j-i)\geq 0
\end{align*}
\end{proof}
In classical measure theory one often identifies integrable functions that agree almost everywhere and the normed space $L^1$ is then defined modulo this equivalence relation. The positive, constructive counterpart is to identify functions in the complete extension of an integration space that agree on a full set. Proposition 2.12 in \cite[p. 227]{BB85} then tells us that we can define the 1-norm modulo this equality. Since, we don't have recourse to the notion of a full set in a predicative setting, we have to introduce the 1-norm a bit differently.
\begin{theorem}\label{pre-int norm}
The relation $=_{\int}$, defined for $i,j\in I$ by
\begin{align*}
i=_{\int}j \;:\Leftrightarrow\; \int\abs{i-j}=0
\end{align*}
is an equivalence relation on $I$ and the assignment routine $\int:(I,=_{\int})\rightsquigarrow\mathbb{R}$ given by $i\mapsto\int i$ is a function. Moreover, the functions $\cdot$ and $+$ make $(I,=_{\int})$ into a $\mathbb{R}$-vector space with neutral element $0\cdot p$, where $p\in I$ is s.t. $\int p=1$, which exists by (PMS3). The function
\begin{align*}
\norm{\_}_1:\;I&\rightarrow\mathbb{R}_{\geq 0} \\
\norm{i}_1&:=\int\abs{i}
\end{align*} 
Defines a norm on $(I,=_{\int};\;\cdot,+,0\cdot p)$.
\end{theorem}

\begin{proof}
First, we have to check that $=_{\int}$ defines an equivalence relation on $I$. Reflexivity follows by (\ref{pre-1.4}) and from the fact that $f_{\abs{i-i}}=\abs{f_i-f_i}=0$ for all $i\in I$. For $i,j\in I$ we have $\abs{f_i-f_j}\eqpf\abs{f_j -f_i}$ and thus $\abs{i-j}=_I\abs{j-i}$, since $\bm\Lambda$ is an $I$-set of partial functions. This establishes the symmetry of $=_{\int}$. For transitivity let $i,j,k\in I$ s.t $i=_{\int}j$ and $j=_{\int}k$. For any $x\in\lzi\cap\lzj\cap\lzk$ we have that
\begin{align*}
f_{\abs{i-j}}(x)=\abs{f_i(x)-f_k(x)+f_k(x)-f_j(x)}\leq f_{\abs{i-k}}(x)+f_{\abs{j-k}}(x)
\end{align*}
and thus by lemma (\ref{pre-1.4.2})
\begin{align*}
0\leq\int\abs{i-j}=\int\abs{i-k+k-j}\leq\int\abs{i-k}+\int\abs{k-j}=0
\end{align*}
which means that $i=_{\int}j$. For $i=_{\int}j$, we have $0\leq\abs*{\int i-\int j}\leq\int\abs{i-j}=0$ by lemma (\ref{pre-1.5}), and hence $\int i=\int j$ which proves that $\int:(I,=_{\int})\rightsquigarrow\mathbb{R}$ is a function.

It follows directly from (PIS1) that $+,\cdot$ satisfy the associativity, commutativity and distributivity laws of a vector-space on the set $(I,=_I)$. Since for all $i,j\in I$ if $i=_I j$ then $\int i={\int}j$ by lemma (\ref{pre-1.4}), it follows that these laws also hold for $(I,=_{\int})$. It remains to check that $0\cdot p$ is indeed a neutral element that and additive inverses exist. But this follows from $i+0\cdot p=_{\int}i$ and $i-i=_{\int}0\cdot p$ for all $i\in I$, which we get using lemma (\ref{pre-1.4}).

It remains to check that $\norm{\_}_1$ defines a norm $(I,=_{\int})$. Let $i\in I$, we have to show that
\begin{align*}
\norm{i}_1=0 \;\Leftrightarrow\; i=_{\int}0\cdot p
\end{align*}
Using the fact that $\bm\Lambda$ is an $I$-set of functions, we get that $\abs{i-0\cdot p}=_I\abs{i}-0\cdot p$, which by lemma (\ref{pre-1.4.2}) and (PIS1) gives us the desired result. The other norm axioms can be obtained similarly.
\end{proof}

\subsection{The pre-integration space of canonically integrable functions}
In our predicative setting we don't have that the totality of integrable functions forms a set. However, in Bishop-Cheng measure theory two integrable functions are identified if they agree on a full set. Each integrable function $f$ has a representation $(f_n)_n$ and agrees with the \emph{canonically integrable function}\footnote{This terminology is due to Bas Spitters, see \cite[p. 24]{Spitters:phd}.} $\sum_n f_n$ on a full set, namely the domain of the canonically integrable function. To construct the complete extension as the metric completion w.r.t. the 1-norm it thus suffices to only consider the canonically integrable functions. In this section we will show that we can actually construct the set of canonically integrable functions and construct the corresponding pre-integration space that will be the complete extension.

\begin{definition}
Let  $(X,I,L,\int )$ be an pre-integration space. We define the set of \emph{representations} to be the totality
\begin{align*}
I_1:=\bigg\{\;\alpha\in\mathbb{F}(\mathbb{N},I) \;:\; \sum_{n=1}^\infty\int\abs{\alpha(n)}\text{ exists}\;\bigg\}
\end{align*}
together with the equality
	\begin{align*}
	\alpha =_{I_1}\beta \;:\Leftrightarrow\; \Big(F_\alpha,e_{F_\alpha},\sum_n f_{\alpha(n)}\Big)\eqpf \Big(F_\beta,e_{F_\beta},\sum_n f_{\beta(n)}\Big)
	\end{align*}
where for $\alpha,\beta\in I_1$
	\begin{align*}
	F_\alpha:=\bigg\{\; x\in\bigcap_n\lambda_0\big(\alpha(n)\big)\;:\; \sum_n\abs{f_{\alpha(n)}(x)}\text{ exists}\;\bigg\}
	\end{align*}
with embedding $e_{F_\alpha}:F_\alpha\hookrightarrow X$ induced by the embeddings $\varepsilon_{\alpha(n)}$ and $(F_\beta,e_{F_\beta})$ is defined accordingly. We define the set of \emph{canonically integrable functions} to be the $I_1$-family of partial functions $\bm\Lambda_1=(\nu_0,\nu_1,E,G)$ given by
\begin{itemize}
\item $\nu_0:I_1\rightsquigarrow\vnull$ is the assignment routine given by $\nu_0(\alpha):=F_\alpha$.
\item $E:\bigcurlywedge_{\alpha\in I_1}\mathbb{F}(\nu_0(\alpha),X)$ is the dependent assignment routine that maps each $\alpha\in I_1$ to the embedding $e_\alpha:=e_{F_\alpha}$.
\item $\nu_1:\bigcurlywedge_{(\alpha,\beta)\in D(I_1)}\mathbb{F}(\nu_0(\alpha),\nu_0(\beta))$ is the dependent assignment routine that maps each $(\alpha,\beta)\in D(I_1)$ to the function $\nu_{\alpha\beta}:\nu_0(\alpha)\rightarrow\nu_0(\beta)$ s.t. $\nu_{\alpha\alpha}:=\id_{\nu_0(\alpha)}$ and $(\nu_{\alpha\beta},\nu_{\beta\alpha})$ witnesses the equality $\alpha=_{I_1}\beta$.\footnote{Here we use the fact that two functions witnessing an equality of subsets are equal.}
\item $G:\bigcurlywedge_{\alpha\in I_1}\mathbb{F}(\nu_0(\alpha),\mathbb{R})$ is the dependent assignment routine that maps each $\alpha\in I_1$ to the function $g_\alpha:\nu_0(\alpha)\rightarrow\mathbb{R}$ that is given by $g_\alpha(x):=\sum_{n=1}^\infty g_{\alpha(n)}(x)$ for each $x\in\nu_0(\alpha)$.
\end{itemize}
\end{definition}

\begin{remark}\label{basic prop of int fcts}
It follows directly from the definition of $=_{I_1}$ that $\bm\Lambda_1$ is an $I_1$-set of partial functions. Furthermore the assignment routine 
\begin{align*}
h:I&\rightsquigarrow I_1 \\
i&\mapsto\big(i,0\cdot i,0\cdot i, ...\big)
\end{align*}
defines an embedding since for $i,j\in I$ we have
\begin{align*}
i=_I j \;&\Leftrightarrow\; f_i\eqpf f_j \\
&\Leftrightarrow \bigg(\lzi,\varepsilon_i, f_i+\sum_{n=2}^\infty 0\cdot f_i\bigg)\eqpf \bigg(\lzj,\varepsilon_j, f_j+\sum_{n=2}^\infty 0\cdot f_j\bigg) \\
&\Leftrightarrow h(i)=_{I_1} h(j)
\end{align*}
since one can easily verify that $\lzi\eqsubset\nu_0(h(i))$ and $\lzj\eqsubset\nu_0(h(j))$.

Let $\alpha,\beta\in I_1$ and $a\in\mathbb{R}$. We can define the following functions 
\begin{align*}
&\_+_1\_:I_1\times I_1\rightarrow I_1 \\
&\;\;\alpha +_1\beta :=\mathsmaller{\mathsmaller{\big(\;\alpha(1),\;\beta(1),\;\alpha(2),\;\beta(2),...\big)}}, \\
&\_\cdot_1\_:\mathbb{R}\times I_1 \rightarrow I_1 \\
&\;\;a\cdot_1\alpha:=\mathsmaller{\mathsmaller{\big(\;a\cdot\alpha(1),\;a\cdot\alpha(2),...\big)}}, \\
&\abs{\_}_1:I_1 \rightarrow I_1 \\
&\;\;\abs{\alpha}_1:=\mathsmaller{\mathsmaller{\big(\;\abs{\alpha(1)},\;\alpha(1),\;(-1)\cdot\alpha(1),\;\abs{\alpha(1)+\alpha(2)}-\abs{\alpha(1)},\;\alpha(2),\;(-1)\cdot\alpha(2),\;\abs{\alpha(1)+\alpha(2)+\alpha(3)}-\abs{\alpha(1)+\alpha(2)},...\big)}}, \\
&\wedge_1^1:I_1 \rightarrow I_1 \\
&\;\;\wedge_1^1(\alpha):=\mathsmaller{\mathsmaller{\big(\;1\wedge \alpha(1),\;\alpha(1),\;(-1)\cdot\alpha(1),\;1\wedge(\alpha(1)+\alpha(2))-1\wedge(\alpha(1)),\;\alpha(2),\;(-1)\cdot\alpha(2),\;1\wedge(\alpha(1)+\alpha(2)+\alpha(3))-1\wedge(\alpha(1)+\alpha(2)),...\big)}}
\end{align*}
See \cite[p. 224]{BB85} to check that
\begin{itemize}
\item $g_{\alpha +_1\beta}\eqpf g_\alpha + g_\beta$
\item $g_{a\cdot_1\alpha}\eqpf a g_\alpha$
\item $g_{\abs{\alpha}_1}\eqpf\abs{g_\alpha}$
\item $g_{\wedge_1^1(\alpha)}\eqpf g_\alpha\wedge 1$
\end{itemize}
Note that for construction of these sequences no choice principles were needed. Furthermore we get that these maps are compatible with $h:I\hookrightarrow I_1$, i.e. for $i,j\in I$ and $a\in\mathbb{R}$ we have that
\begin{itemize}
\item $h(i+j)=_{I_1}h(i) +_1 h(j)$
\item $h(a\cdot i)=_{I_1}a\cdot_1 h(i)$
\item $h(\abs{i})=_{I_1}\abs{h(i)}_1$
\item $h(\wedge_1(i))=_{I_1}\wedge_1^1(h(i))$
\end{itemize}
In the following we will thus drop the subscript  $1$ for all of these functions.

Finally the assignment routine $\int_1:I_1\rightsquigarrow\mathbb{R}$ given by $\int_1\alpha:=\sum_n\int\alpha(n)$ is a function. To see this let $\alpha=_{I_1}\beta$ and assume that $\int_1\alpha\neq_{\mathbb{R}}\int_1\beta$. Take $a\in\mathbb{R}$ and $N\in\mathbb{N}$ s.t.
\begin{itemize}
\item $0<3a<\abs*{\sum_n\int \alpha(n)-\sum_n\int\beta(n)}$
\item $\sum_{n=N+1}^\infty \int\abs{\alpha(n)} < a$
\item $\sum_{n=N+1}^\infty\int\abs{\alpha(n)-\beta(n)}<a$
\end{itemize}
Then 
\begin{align*}
\sum_{n=N+1}^\infty\int\abs{\alpha(n)}&\;+\sum_{n=N+1}^\infty\int\abs{\alpha(n)-\beta(n)} <2a \\[1em]
&<\abs*{\sum_{n=1}^\infty\int\alpha(n)-\sum_{n=1}^\infty\int\beta(n)} -a \\[1em]
&<\abs*{\sum_{n=1}^N\int\big(\alpha(n)-\beta(n)\big)}\;+\sum_{n=N+1}^\infty\int\abs{\alpha(n)-\beta(n)}-a \\[1em]
&<\abs*{\sum_{n=1}^N\int\big(\alpha(n)-\beta(n)\big)}\leq\int\abs*{\sum_{n=1}^N\big(\alpha(n)-\beta(n)\big)}
\end{align*}
Thus by (PIS2) for the pre-integration space $(X,I,\bm\Lambda,\int)$ there is a $x\in\big(\bigcap_n\lambda_0(\alpha(n))\big)\cap\big(\bigcap_n\lambda_0(\beta(n))\big)$ s.t. $\sum_n\abs{f_{\alpha(n)}(x)}$ and $\sum_n\abs{f_{\alpha(n)}(x)-f_{\beta(n)}(x)}$ both exist and we have
\begin{align*}
\sum_{n=N+1}^\infty\abs{f_{\alpha(n)}(x)}\; +\sum_{n=N+1}^\infty\abs{f_{\alpha(n)}(x)-f_{\beta(n)}(x)}<\abs*{\sum_{n=1}^N\big(f_{\alpha(n)}(x)-f_{\beta(n)}(x)\big)}
\end{align*}
Hence, $\sum_n\abs{f_{\beta(n)}(x)}$ exists and thus
\begin{align*}
0\leq\sum_{n=N+1}^\infty\abs{f_{\alpha(n)}(x)}&<\abs*{\sum_{n=1}^N\big(f_{\alpha(n)}(x)-f_{\beta(n)}(x)\big)}\;-\sum_{n=N+1}^\infty\abs{f_{\alpha(n)}(x)-f_{\beta(n)}(x)} \\[1em]
&\leq \abs*{\sum_{n=1}^\infty\big(f_{\alpha(n)}(x)-f_{\beta(n)}(x)\big)}
\end{align*}
But this is a contradiction since $\alpha=_{I_1}\beta$ implies that $\sum_n f_{\alpha(n)}(x)=_{\mathbb{R}}\sum_n f_{\beta(n)}(x)$. It is also clear that $\int_1 h(i)=\int i$ for all $i\in I$. We will thus drop the subscript $1$ here as well.
\end{remark}

We now prove some basic properties of $I_1$ and $\bm\Lambda_1$. Again, all of the arguments are taken from section 2 of chapter 6 of \cite{BB85}.

\begin{lemma}\label{pre-2.3}
$\forall\alpha\in I_1:\;\abs*{\int\alpha}\leq\int\abs{\alpha}$
\end{lemma}
\begin{proof}
From the definition we get that
\begin{align*}
\int\abs{\alpha}=\lim_{n\rightarrow\infty}\int\abs*{\sum_{k=1}^n\alpha(k)}
\end{align*}
Now assume that $\abs*{\int\alpha}>\int\abs{\alpha}$, then there is $n$ large enough s.t.
\begin{align*}
\int\abs*{\sum_{k=1}^n\alpha(k)} <\abs*{\sum_{k=1}^n\int\alpha(k)}=\abs*{\int\sum_{k=1}^n\alpha(k)}
\end{align*}
This is a contradiction to lemma (\ref{pre-1.5}).
\end{proof}

\begin{lemma}\label{pre-2.7}
Let $\alpha\in I_1$ be s.t. for all $x\in \nu_0(\alpha)$ we have $g_\alpha(x)\geq 0$, then $\int\alpha\geq 0$.
\end{lemma}
\begin{proof}
For $i\in I$ define
\begin{align*}
i^+ &:= \frac{1}{2}\cdot\abs{i}+\frac{1}{2}\cdot i \\
i^- &:= \frac{1}{2}\cdot\abs{i}-\frac{1}{2}\cdot i
\end{align*}
Then $f_{i^+}\eqpf f_i^+$ and $f_{i^-}\eqpf f_i^-$. Let $\alpha\in I_1$, then by assumption we have that $\sum_n\int\alpha(n)^+$ and $\sum_n\int \alpha(n)^-$ exist and that
\begin{align*}
\int\alpha:=\sum_{n=1}^\infty\int {\alpha(n)}=\sum_{n=1}^\infty\int {\alpha(n)}^+ -\sum_{n=1}^\infty\int {\alpha(n)}^-
\end{align*}
Furthermore, we have for $x\in \nu_0(\alpha)$ that $\sum_n \vert f_{\alpha(n)}(x)\vert$ exists if and only if $\sum_n f_{\alpha(n)}^+(x)$ and $\sum_n f_{\alpha(n)}^-(x)$ exist, in which case
\begin{align*}
\sum_{n=1}^\infty f_{\alpha(n)}^+(x) -\sum_{n=1}^\infty f_{\alpha(n)}^-(x)=g_{\alpha}(x)\geq 0
\end{align*}
Now, suppose that $\int\alpha< 0$ and take $\varepsilon>0$ s.t. $\sum_n\int {\alpha(n)}^+<\sum_n\int {\alpha(n)}^- -2\varepsilon$. Also take $N\in\mathbb{N}$ large enough s.t. $\sum_{n=N+1}^\infty\int {\alpha(n)}^-<\varepsilon$. We have that
\begin{align*}
\sum_{n=1}^\infty\int {\alpha(n)}^+ +\sum_{n=N+1}^\infty\int \alpha(n)^- < \sum_{n=1}^N\int {\alpha(n)}^- +2\sum_{n=N+1}^\infty\int {\alpha(n)}^- -2\varepsilon 
<\sum_{n=1}^N\int {\alpha(n)}^-
\end{align*}
Hence, by (PIS2) there is a $x\in \bigcap_n(\lambda_0(\alpha(n))$ s.t. $\sum_n f_{\alpha(n)}^+(x)$ and $\sum_n f_{\alpha(n)}^-(x)$ exist , and thus also $\sum_n \vert f_{\alpha(n)}(x)\vert$ converges, i.e. $x\in\nu_0(\alpha)$, and we have
\begin{align*}
\sum_{n=1}^\infty f_{\alpha(n)}^+(x)<\sum_{n=1}^N f_{\alpha(n)}^-(x)\leq \sum_{n=1}^\infty f_{\alpha(n)}^-(x)
\end{align*}
it follows that $g_\alpha(x)=\sum_{n} f_{\alpha(n)}(x)<0$, which is a contradiction and thus $\int\alpha\geqslant 0$. 
\end{proof}

\begin{lemma}\label{pre-2.8}
Let $\alpha,\beta\in I_1$ be s.t. for all $x\in\nu_0(\alpha)\cap\nu_0(\beta)$ we have that $g_\alpha(x)\leq g_\beta(x)$, then $\int \alpha\leq\int\beta$.
\end{lemma}
\begin{proof}
Consider $\beta-\alpha\in I_1$, then $\nu_0(\beta-\alpha)\eqsubset\nu_0(\alpha)\cap\nu_0(\beta)$ and thus $g_{\beta-\alpha}\geq 0$. Hence, by lemma (\ref{pre-2.7})
\begin{align*}
\int \beta-\int \alpha=\int (\beta-\alpha)\geq 0
\end{align*}
\end{proof}

\begin{lemma}\label{pre-2.14}
There is a function $\phi:I_1\rightarrow\mathbb{N}\rightarrow I_1$ s.t. for every $\alpha\in I_1$ and $n\in\mathbb{N}$ we have  that $\phi(\alpha,n):=\beta=_{I_1}\alpha$ and $\sum_k\int\abs{\beta(k)}\leq\int\abs{\alpha}+2^{-n}$.
\end{lemma}
\begin{proof}
Let $\alpha\in I_1$ and $n\in\mathbb{N}$. Let  $M:\mathbb{N}\rightarrow\mathbb{N}$ be the modulus of converengence of the sequence $\big(\sum_{k=1}^m\int\abs{\alpha(k)}\big)_{m=1}^\infty$ and $N:=M(n+1)$, i.e. we have $\sum_{k=N+1}^\infty \int\abs{\alpha(k)} \leq 2^{-(n+1)}$. Let $\beta\in \mathbb{F}(\mathbb{N},I)$ be given by
\begin{align*}
\Big(\sum_{k=1}^N {\alpha(k)},\alpha(N+1),\alpha(N+2),...\Big)
\end{align*}
Then $\sum_k\int\abs{\beta(k)}\leq\sum_k\int\abs{\alpha(k)}$ and thus $\beta\in I_1$, i.e. $\phi(\alpha,n):=\beta$ is well-defined.  Furthermore, we clearly have $\nu_0(\beta)\eqsubset\nu_0(\alpha)$ and $g_\beta\eqpf g_\alpha$.

For $m>N$ we obtain, by using lemma (\ref{pre-2.8}), that
\begin{align*}
\sum_{k=1}^\infty\int\abs{\beta(k)} &\leq \int \abs{\beta(1)} + \frac{1}{2^{n+1}} \\[1em]
&\leq \int\Bigg(\abs*{\sum_{k=1}^m \alpha(k)}\Bigg) + \int \Bigg(\abs*{\sum_{k=1}^N \alpha(k)}-\abs*{\sum_{k=1}^m {\alpha(k)}}\Bigg) +\frac{1}{2^{n+1}} \\[1em]
&\leq\int\Bigg(\abs*{\sum_{k=1}^m {\alpha(k)}}\Bigg) + \int \Bigg(\abs*{\sum_{k=N+1}^m {\alpha(k)}}\Bigg) +\frac{1}{2^{n+1}}  \\[1em]
&\leq \int\Bigg(\abs*{\sum_{k=1}^m {\alpha(k)}}\Bigg) +\sum_{k=N+1}^m\int\abs{ {\alpha(k)}} +\frac{1}{2^{n+1}} \\[1em]
&< \int\Bigg(\abs*{\sum_{k=1}^m {\alpha(k)}}\Bigg)+2^{-n}
\end{align*}
By the proof of lemma (\ref{pre-2.3}) we have that
\begin{align*}
\lim_{m\rightarrow\infty}\int\Bigg(\abs*{\sum_{k=1}^m {\alpha(k)}}\Bigg)=\int\abs{\alpha}
\end{align*}
which finishes the proof.
\end{proof}

We are now able to prove the predicative, constructive version of Lebesgue's series theorem. The proof generally follows the proof theorem 2.15 in \cite{BB85}, but we have to be a bit more cautious, since we don't have a notion of a full set at hand and thus have to work around arguments that rely on properties of full sets. Again, using lemma (\ref{pre-2.14}) instead of a more straightforward translation of 2.14 in \cite{BB85} allows us to avoid countable choice.
\newpage
\begin{theorem}\label{pre-2.15}
Let $\Gamma\in\mathbb{F}(\mathbb{N},I_1)$ be s.t. $\sum_n\int\abs{\Gamma(n)}$ exists. Then there is an $\alpha\in I_1$ s.t.
\begin{align*}
&\nu_0(\alpha)\subseteq \bigg\{\; x\in\bigcap_{n=1}^\infty\nu_0\big(\Gamma(n)\big)\;:\; \sum_{n=1}^\infty\abs{g_{\;\Gamma(n)}(x)}\text{ exists}\;\bigg\} \\[1em]
&\text{and}\quad \forall x\in\nu_0(\alpha):\; g_\alpha(x)=\sum_{n=1}^\infty g_{\;\Gamma(n)}(x)
\end{align*}
Furthermore, for any $\alpha\in I_1$ fulfilling the above condition we have
\begin{align*}
\lim_{N\rightarrow\infty}\int\abs{\alpha-\sum_{n=1}^N\Gamma(n)} =0
\end{align*}
\end{theorem}
\begin{proof}
Let $A:=\{ x\in\bigcap_{n=1}^\infty\nu_0\big(\Gamma(n)\big): \sum_{n=1}^\infty\abs{g_{\;\Gamma(n)}(x)}\text{ exists}\} $ and $i_A:A\hookrightarrow X$ be the embedding induced by the $e_{\Gamma(n)}$. Let 
\begin{align*}
g_A:A&\rightarrow\mathbb{R} \\
x &\mapsto \sum_{n=1}^\infty g_{\Gamma(n)}(x)
\end{align*}
For each $n\in\mathbb{N}$ let $\beta_n:=\phi(\Gamma(n),n)$ with $\phi$ as in lemma (\ref{pre-2.14}), i.e. $\beta_n\in I_1$ is s.t. for all $n\in \mathbb{N}$ we have $\beta_n=_{I_1}\Gamma(n)$ and
\begin{align*}
\sum_{k=1}^\infty\int \abs{\beta_n(k)}<\int \abs{\Gamma(n)}\;+\; 2^{-n}
\end{align*}
It follows that $\sum_{n=1}^\infty\sum_{k=1}^\infty\int \vert f_{\beta_n(k)}\vert$ exists. Let
\begin{align*}
B:=\bigg\{\; x\in\bigcap_{n\in\mathbb{N}}\bigcap_{k\in\mathbb{N}}\lambda_0(\beta_n(k)) \;:\; \sum_{n=1}^\infty\sum_{k=1}^\infty \abs{ f_{\beta_n(k)}(x)} \text{ exists}\;\bigg\}
\end{align*}
and let the bijection $\varphi:\mathbb{N}\rightarrow\mathbb{N}\times\mathbb{N}$ be given as in section (\ref{reals}) and let $\alpha\in\mathbb{F}(\mathbb{N},I)$ be given by
 $\alpha(n):=\beta_{\mathsf{pr}_1(\varphi(n))}\big(\mathsf{pr}_2(\varphi(n))\big)$. Then by lemma (\ref{doubleseries}) we have that $\sum_{n=1}^\infty\int \abs{\alpha(n)}=\sum_{n=1}^\infty\sum_{k=1}^\infty\int \vert \beta_n(k)\vert$ exists (i.e. $\alpha\in I_1$).
Moreover, we clarly have functions $\psi$ and $\psi^{-1}$ s.t.
\begin{align*}
(\psi,\psi^{-1}):\bigcap_{n\in\mathbb{N}}\lambda_0(\alpha(n))\eqsubset \bigcap_{n\in\mathbb{N}}\bigcap_{k\in\mathbb{N}}\lambda_0(\beta_n(k))
\end{align*}
By lemma (\ref{doubleseries}), for any $x\in\bigcap_{n\in\mathbb{N}}\lambda_0(\alpha(n))$ we have that $\sum_{n=1}^\infty\vert f_{\alpha(n)}(x)\vert$ exists if and only if $\sum_{n=1}^\infty\sum_{k=1}^\infty \abs{ f_{\beta_n(k)}(\psi(x))}$ exists, while for any $x\in\bigcap_{n\in\mathbb{N}}\bigcap_{k\in\mathbb{N}}\lambda_0(\beta_n(k))$ we have that $\sum_{n=1}^\infty\sum_{k=1}^\infty \abs{ f_{\beta_n(k)}(x)}$ exists if and only if $\sum_{n=1}^\infty\vert f_{\alpha(n)}(\psi^{-1}(x))\vert$ exists. It follows that
\begin{align*}
\nu_0(\alpha):=\bigg\{\; x\in\bigcap_{n\in\mathbb{N}}\lambda_0(\alpha(n)) \;:\; \sum_{n=1}^\infty\vert f_{\alpha(n)}(x)\vert \text{ exists}\;\bigg\}\; \eqsubset \;B
\end{align*}
and furthermore, again by lemma (\ref{doubleseries}), the following diagrams commute
\[
\begin{tikzcd}
\nu_0(\alpha)\arrow[rr, "\psi\big|_{\nu_0(\alpha)}", bend left]\arrow[rdd,"g_\alpha=\sum_n f_{\alpha(n)}"']& & B\arrow[ll, "\psi^{-1}\big|_B"', bend left]\arrow[ldd, "\sum_n\sum_k f_{\beta_n(k)}"] \\
\\
&\mathbb{R} 
\end{tikzcd}
\]
For $n\in\mathbb{N}$ we have an embedding $e_n:\nu_0(\beta_n)\hookrightarrow\nu_0(\Gamma(n))$ (in fact $e_n=\nu_{\beta_n\Gamma(n)}$) and thus by the definition of $B$ we have an embedding $e_n':B\hookrightarrow \nu_0(\beta_n)$ s.t. the following diagram commutes
\[
\begin{tikzcd}
B \arrow[r, "e_n'", hook]\arrow[d, hook] & \nu_0(\beta_n) \arrow[d, "e_n", hook] \arrow[rd, "\sum_k f_{\beta_n(k)}"] \\
X &\nu_0(\Gamma(n)) \arrow[l, "e_{\Gamma(n)}", left hook->]\arrow[r, "g_{\Gamma(n)}"'] &\mathbb{R}
\end{tikzcd}
\]
Moreover, for $m,n\in\mathbb{N}$ we have $e_{\Gamma(n)}\circ e_n\circ e_n'=_{\mathsmaller{\mathbb{F}(B,X)}} e_{\Gamma(m)}\circ e_m\circ e_m'$ and thus the maps $e_n\circ e_n'$ give us an embedding $e':B\hookrightarrow\bigcap_{n\in\mathbb{N}}\nu_0(\Gamma(n))$. By the commutativity of the above diagram, we get that for $x\in B$
\begin{align*}
\sum_{n=1}^\infty\abs{g_{\Gamma(n)}\big( e'(x)\big)}\leq \sum_{n=1}^\infty\sum_{k=1}^\infty \abs{f_{\beta_n(k)}(x)} <\infty
\end{align*}
Hence, $e':B\hookrightarrow A$ and following diagram commutes
\[
\begin{tikzcd}
&&X \\
\nu_0(\alpha) \arrow[r, equal]\arrow[rdd, bend right, "g_\alpha=\sum_n f_{\alpha(n)}"'] &B\arrow[rr, "e' ", hook]\arrow[ru, "i_B", hook]\arrow[dd, "\sum_n\sum_k f_{\beta_n(k)}"] &&A\arrow[lu, "i_A"', left hook->]\arrow[lldd, bend left, "g_A=\sum_n g_{\Gamma(n)}"] 
\\
\\
&\mathbb{R}
\end{tikzcd}
\]
This proves the first part of the theorem.

Now fix $N\in\mathbb{N}$ and let $\alpha\in I_1$ s.t. that it satisfies the conditions of the first part of the theorem. Let $\gamma:=\abs*{\alpha-\sum_{n=1}^N\Gamma(n)}\in I_1$. Now let $\delta\in \mathbb{F}(\mathbb{N},I)$ be an enumeration of the terms 
\[
\begin{matrix}
-\gamma(1) &-\gamma(2) &-\gamma(3) & \cdots \\
\abs{\beta_{N+1}(1)} &\abs{\beta_{N+1}(2)} &\abs{\beta_{N+1}(3)} &\cdots \\
\abs{\beta_{N+2}(1)} &\abs{\beta_{N+2}(2)} &\abs{\beta_{N+2}(3)} &\cdots \\
\vdots &\vdots &\vdots &\ddots
\end{matrix}
\]
into a single sequence using the bijection $\varphi:\mathbb{N}\rightarrow\mathbb{N}\times\mathbb{N}$. Then, by lemma (\ref{doubleseries}) we have that 
\begin{align*}
\sum_{n=1}^\infty\int\abs{\delta(n)}=\sum_{n=1}^\infty\int\abs{\gamma(n)} \;+\; \sum_{n=N+1}^\infty\sum_{k=1}^\infty\int\abs{\beta_n(k)}<\infty
\end{align*}
(i.e. $\delta\in I_1$) and for each $x\in \nu_0(\delta)$ we have that
\begin{align*}
\sum_{n=1}^\infty f_{\delta(n)}(x)&=\sum_{n=N+1}^\infty\sum_{k=1}^\infty\abs{f_{\beta_n(k)}(x)}-\sum_{m=1}^\infty f_{\gamma(m)}(x) \\[1em]
&\geq \sum_{n=N+1}^\infty\abs{g_{\Gamma(n)}(x)}-g_\gamma(x) \\[1em]
&\geq \sum_{n=N+1}^\infty\abs{g_{\Gamma(n)}(x)}-\abs*{g_\alpha(x)-\sum_{n=1}^N g_{\Gamma(n)}(x)} \\
&\geq 0
\end{align*}
By lemma (\ref{pre-2.7}) it follows that $\int\delta\geq 0$ and hence
\begin{align*}
0&\leq \int\Bigg(\abs*{\alpha-\sum_{n=1}^N \Gamma(n)}\Bigg)=\int\gamma =\sum_{n=1}^\infty\int\gamma(n)\\[1em]
&\leq\sum_{n=N+1}^\infty\sum_{k=1}^\infty\int\abs{\beta_n(k)}\leq \sum_{n=N+1}^\infty\Big(\int\abs{\Gamma(n)}+2^{-n}\Big)
\end{align*}
and the last expression converges to $0$ for $N\rightarrow\infty$, which finishes the proof.
\end{proof}

\begin{cor}\label{cor to pre-2.15}
For  all $\alpha\in I_1$ we have
\begin{align*}
\lim_{N\rightarrow\infty}\int\abs{\alpha-\sum_{n=1}^N\alpha(n)} =0
\end{align*}
\end{cor}
\begin{proof}
It is immediate to see that $\alpha$ satisfies  the condition of the first part of theorem (\ref{pre-2.15}) for the sequence $\Gamma(n):=\alpha(n)$, where $\alpha(n)$ is identified with its image under the embedding $h$.
\end{proof}

With Lebesgue's series theorem at hand we can now show that the canonically integrable functions form a pre-integration space and as such the complete extension of the pre-integration space $(X,I,\bm\Lambda,\int)$.

\begin{theorem}\label{pre-2.18}
$(X,I_1,\bm\Lambda_1,\int)$ is a pre-integration space.
\end{theorem}
\begin{proof}
We have already established (PIS1) in remark (\ref{basic prop of int fcts}). Let $\Gamma\in\mathbb{F}(\mathbb{N},I_1)$ be such that $\ell:=\sum_n\int g_{\Gamma(n)}$ exists and for all $n\in\mathbb{N}$ we have that $g_{\Gamma(n)}\geq 0$, let $\alpha\in I_1$ be such that $\ell<\int\alpha$. Without loss of generality we can assume that $g_\alpha\geq 0$. Otherwise let $\alpha':=\abs{\alpha}+\alpha$ and let $\Gamma'\in\mathbb{F}(\mathbb{N},I_1)$ be given by $\Gamma'(1):=\abs{\alpha}$ and $\Gamma'(n):=\Gamma(n-1)$ for $n\geq 2$. Then $g_{\alpha'}\geq 0$ and $\sum_n\int\Gamma'(n)=\int\abs{\alpha}+\ell<\int\alpha'$. Showing (PIS2) for $\alpha'$ and $\Gamma'$ also gives the desired result for $\alpha$ and $\Gamma$.

Let $\varepsilon>0$ be s.t. $\ell+ 3\varepsilon<\int\alpha$ and $p\in\mathbb{N}$ s.t. $2^{-p}<\varepsilon$. For $n\in\mathbb{N}$ let $\beta_n:=\phi\big(\Gamma(n),n+p\big)$ with $\phi$ as defined in lemma (\ref{pre-2.14}), i.e. $\beta_n=_{I_1}\Gamma(n)$ and
\begin{align*}
\sum_{k=1}^\infty\int\abs{\beta_n(k)}<\int\abs{\Gamma(n)}+\frac{1}{2^{n+p}}<\int\Gamma(n)+\frac{\varepsilon}{2^n}
\end{align*}
Let $N\in\mathbb{N}$ be large enough so that 
\begin{itemize}
\item $\int\alpha<\int\abs{\sum_{k=1}^N\alpha(k)} + \varepsilon$
\item $\sum_{k\geq N+1}\int\abs{\alpha(k)}<\varepsilon$
\end{itemize}
Such a $N$ exists since $\alpha\in I_1$ and
\begin{align*}
\lim_{N\rightarrow\infty}\int\abs*{\sum_{k=1}^N\alpha(k)}=\int\abs{\alpha}=\int\alpha
\end{align*}
Then $\sum_n\sum_k\int\abs{\beta_n(k)}$ exists and
\begin{align*}
\sum_{n=1}^\infty\sum_{k=1}^\infty\int\abs{\beta_n(k)}+\sum_{k=N+1}^\infty\int\abs{\alpha}<\sum_{n=1}\int\Gamma(n)+2\varepsilon <\int\alpha-\varepsilon<\int\abs*{\sum_{k=1}^N\alpha(k)}
\end{align*}
Since $(X,I,L\int)$ is an integration space there is a $x\in\big(\bigcap_n\lambda_0(\alpha(n))\big)\cap\big(\bigcap_n\nu_0(\Gamma(n))\big)$ s.t.
\begin{align*}
\sum_{n=1}^\infty\sum_{k=1}^\infty\abs{f_{\beta_n(k)}(x)}+\sum_{k=N+1}^\infty\abs{f_{\alpha(k)}(x)}<\abs*{\sum_{k=1}^N f_{\alpha(k)}(x)}
\end{align*}
it follows that $x\in \nu_0(\alpha)\cap\big(\bigcap_n\nu_0(\Gamma(n))\big)$, $\sum_n\abs{g_{\Gamma(n)}(x)}$ exists and
\begin{align*}
\sum_{n=1}^\infty g_{\Gamma(n)}(x)&=\sum_{n=1}^\infty\sum_{k=1}^\infty f_{\beta_n(k)}(x)\leq\sum_{n=1}^\infty\sum_{k=1}^\infty\abs{f_{\beta_n(k)}(x)} \\[1em]
&<\abs*{\sum_{k=1}^N f_{\alpha(k)}(x)}-\abs*{\sum_{k=N+1}^\infty f_{\alpha(k)}(x)} \\[1em]
&\leq\abs*{\sum_{k=1}^\infty f_{\alpha(k)}(x)}=\abs{g_\alpha(x)}=g_\alpha(x)
\end{align*}
This finishes the verification of (PIS2). For (PIS3) let $i\in I$ be such that $\int i=1$, then $\int h(i)=1$.

To verify (PIS4) let $\alpha\in I_1$ and $\varepsilon>0$. By corollary (\ref{cor to pre-2.15}) there is an $i\in I$ s.t. $\int\abs{\alpha-i}\leq\nicefrac{\varepsilon}{3}$. For $r>0$ and $\alpha\in I_1$ define $\alpha\wedge r:=r\cdot(\wedge_1(r^{-1}\cdot\alpha))$. Take $N\in \mathbb{N}$ s.t. $\int i\wedge N >\int i-\nicefrac{\varepsilon}{3}$, then for all $r>0$
\begin{align*}
\abs*{\int \alpha\wedge r -\int i\wedge r}\leq\int\abs{\alpha\wedge r-i\wedge r} \overset{\mathsmaller{\text{ lemma (\ref{pre-2.8})}}}{\leq}\int\abs{\alpha-i}<\frac{\varepsilon}{3}
\end{align*}
Hence for all $n\in\mathbb{N}$
\begin{align*}
\int\alpha\geq\int\alpha\wedge n\geq\int \alpha\wedge N>\int i\wedge N -\frac{\varepsilon}{3}>\int i-\frac{2\varepsilon}{3}>\int \alpha-\varepsilon
\end{align*}
It follows that $\lim_{n\rightarrow\infty}\int\alpha\wedge n=\int\alpha$. Now take $m\in\mathbb{N}$ s.t. $\int \abs{i}\wedge m^{-1}<\nicefrac{\varepsilon}{2}$.  Using the inverse triangle-inequality and lemma (\ref{pre-2.8}) we get that
\begin{align*}
\int\abs{\;\abs{\alpha}\wedge r-\abs{i}\wedge r\;}\leq\int\abs{\;\abs{\alpha}-\abs{i}\;}\leq \int\abs{\alpha-i}<\frac{\varepsilon}{3}
\end{align*}
for all $r>0$, and thus for $n\geq m$
\begin{align*}
0\leq \int \abs{\alpha}\wedge n^{-1}\leq\int\abs{\alpha}\wedge m^{-1}\leq\int\abs{i}\wedge m^{-1}+\frac{\varepsilon}{3}<\frac{\varepsilon}{2}+\frac{\varepsilon}{3}<\varepsilon
\end{align*}
which finishes the verification of (PIS4).
\end{proof}

\begin{theorem}\label{pre-2.16}
The embedding $h:I\hookrightarrow I_1$ is norm preserving and $(I,=_{\int},\norm{\_}_1)$ becomes a dense subspace of $(I_1,=_{\int},\norm{\_}_1)$ through $h$.
\end{theorem}
\begin{proof}
This follows directly from (\ref{cor to pre-2.15}).
\end{proof}

\begin{theorem}\label{pre-2.17}
$I_1$ is complete with respect to $\norm{\_}_1$.
\end{theorem}
\begin{proof}
Let $\Gamma\in\mathbb{F}(\mathbb{N},I_1)$ be a Cauchy-sequence with strictly increasing modulus $M:\mathbb{N}\rightarrow\mathbb{N}$, i.e. for each $p\in\mathbb{N}$ and $n,m\geq M(p)$ we have that $\int\abs{\Gamma(n)-\Gamma(m)}\leq2^{-p}$. Define $\Delta\in\mathbb{F}(\mathbb{N},I_1)$ by $\Delta(1):=\Gamma\big(M(1)\big)$ and $\Delta(n):=\Gamma\big(M(n)\big)-\Gamma\big(M(n-1)\big)$ for $n\geq 2$. Then $\sum_n\int\abs{\Delta(n)}$ exists and by theorem (\ref{pre-2.15}) there is an $\alpha\in I_1$ s.t.
\begin{align*}
\norm*{\alpha-\Gamma\big(M(n)\big)}_1=\int\abs*{\alpha-\sum_{k=1}^n\Delta(k)}\xrightarrow{n\rightarrow\infty}0
\end{align*}
With some modulus of convergence $M'$. Now, let $p\in\mathbb{N}$ and $n:=\big(M'(p+1)\big)\vee\big(M(p+1)\big)$. Then for $m\geq n$ we have
\begin{align*}
\norm*{\alpha-\Gamma(m)}_1\leq \norm*{\alpha-\Gamma(M(n))}_1+\norm*{\Gamma(M(n))-\Gamma(m)}_1\leq\frac{1}{2^{p+1}}+\frac{1}{2^{p+1}}=\frac{1}{2^p}
\end{align*}
i.e. $\Gamma$ converges to $\alpha$ in the norm $\norm{\_}_1$ with modulus 
\begin{align*}
M'':=\big(M'(\_+1)\big)\vee\big(M(\_+1)\big)
\end{align*}
\end{proof}

\chapter{Conclusion}
This paper consists of two parts. The first part is concerned with Bishop's set theory. Following \cite{Petrakis2018} and \cite{Petrakis2019} we introduced an informal (or semi-formal) language to define the basic notions of BST along with those notions we need for BCMT. Moreover we gave a detailed account of the various forms of set-indexed families. The general theory of these families and their application to different areas of constructive mathematics will be studied extensively in \cite{Petrakis2020}. We restricted ourselves to the parts needed for BCMT. 

In the second part saw how we can apply the idea of set-indexed families to constructive measure theory and work towards a predicative reconstruction of BCMT. To that end we introduced the notions of pre-integration and pre-measure space as a predicative counterpart to integration and measure spaces and revisited central parts BCMT building on these two notions. In particular we have given a concrete example of a pre-measure space, namely the pre-measure space of detachable subsets with the dirac measure, we have shown that for any pre-measure space we can construct the pre-integration space of simple functions and we have shown that the complete extension of a pre-integration space can be constructed predicatively by only considering the canonically integrable functions.

Of course, there are a lot of remaining questions that need to be resolved in order to show that the notions of pre-integration and pre-measure space can be applied fruitfully to get a predicative account of BCMT. One important open task is to define the pre-measure space of integrable \emph{sets} induced by the completion of a pre-integration space. The problem here is that we have to introduce two families of complemented subsets, the set of integrable sets and the larger family that takes the place of the totality of all complemented subsets. It seems that there is no canonical way of introducing this larger family. A related open task is to give the appropriate axioms for a \emph{complete} pre-measure space. Again, it seems that there are some non-canonical choices to be made and it would be really interesting to see if all of section 10 of chapter 6 of \cite{BB85} can be made predicative, using our notion of pre-integration and pre-measure-space.

Another big issue is the treatement of measurable functions that play a crucial role in BCMT and especially its applications to e.g. constructive functional analysis. From the predicative  point of view, the totality of measurable functions is not a set since it contains all integrable functions. There have been metric approaches to measurable functions, e.g. in \cite{Spitters:phd} and \cite{Spitters2006} the space of measurable functions is defined as the completion of the space of integrable functions equipped with a certain uniform structure and in \cite{Ishihara2017} the space of measurable functions is the completion of an (elementary) integration space with respect to the pseudo-metric
\begin{align*}
d_m(f,g):=\int\big(\abs{f-g}\wedge 1\big)
\end{align*}
It would be really interesting to see if it is possible to apply our strategy (to define the complete extension predicatively) to the measureable functions, i.e. if one could define a set of `canonically measurable' functions and show that its index-set when equipped with the corresponding topological structure is the completion of the pre-integration space of canonically integrable functions.

Finally, it would be quite desirable to formalize the work presented here. As already mentioned, parts of BCMT are already formalized in the proof assistant Coq, which is impredicative, in \cite{Semeria2019}, but it seems that our predicative account could lead to a formalization that carries a lot of computational content that is otherwise destroyed by impredicativities or the use of choice principles.

\bibliographystyle{alpha}
\bibliography{mastermathverzeichnis}

\begin{thebibliography}{BBDS13}

\bibitem[BB85]{BB85}
E.~Bishop and D.~S. Bridges.
\newblock {\em Constructive Analysis}, volume 279 of {\em Grundlehren der math.
  Wissenschaften}.
\newblock Springer-Verlag, Heidelberg-Berlin-New York, 1985.

\bibitem[BB09]{Berger2009}
Josef Berger and Douglas~S Bridges.
\newblock Rearranging series constructively.
\newblock {\em J. UCS}, 15(17):3160--3168, 2009.

\bibitem[BBDS13]{Berger2013}
Josef Berger, Douglas Bridges, Hannes Diener, and Helmut Schwichtenberg.
\newblock Constructive aspects of {R}iemann's permutation theorem for series.
\newblock {\em arXiv preprint arXiv:1303.7051}, 2013.

\bibitem[BC72]{Bishop1972}
Errett Bishop and Henry Cheng.
\newblock {\em Constructive measure theory}, volume 116.
\newblock American Mathematical Soc., 1972.

\bibitem[Bis67]{Bishop67}
Errett Bishop.
\newblock {\em Foundations of constructive analysis}.
\newblock McGraw-Hill series in higher mathematics. McGraw-Hill, 1967.

\bibitem[Bis70]{Bishop1970}
Errett Bishop.
\newblock Mathematics as a numerical language.
\newblock In {\em Studies in Logic and the Foundations of Mathematics},
  volume~60, pages 53--71. Elsevier, 1970.

\bibitem[Cha21]{Chan2019}
Yuen-Kwok Chan.
\newblock {\em Foundations of Constructive Probability Theory}.
\newblock Encyclopedia of Mathematics and its Applications. Cambridge
  University Press, 2021.

\bibitem[CP02]{Coquand2002}
Thierry Coquand and Erik Palmgren.
\newblock Metric {B}oolean algebras and constructive measure theory.
\newblock {\em Arch. Math. Log.}, 41:687--704, 10 2002.

\bibitem[Ish17]{Ishihara2017}
Hajime Ishihara.
\newblock A constructive theory of integration – a metric approach.
\newblock unpublished note, aug 2017.

\bibitem[Pet19a]{Petrakis2018}
Iosif Petrakis.
\newblock Dependent sums and dependent products in {B}ishop's set theory.
\newblock In P.~Dybjer et~al., editors, {\em TYPES 2018}, volume 130 of {\em
  LIPIcs}, 2019.

\bibitem[Pet19b]{Petrakis2019b}
Iosif Petrakis.
\newblock The pre-measure space of complemented detachable subsets.
\newblock unpublished note, jan 2019.

\bibitem[Pet20]{Petrakis2020}
Iosif Petrakis.
\newblock Families of sets in bishop set theory.
\newblock Habilitation Thesis, LMU, 2020.
\newblock Available at
  \url{https://www.mathematik.uni-muenchen.de/~petrakis/content/Theses.php}.

\bibitem[Pet21]{Petrakis2019}
Iosif Petrakis.
\newblock {Direct spectra of Bishop spaces and their limits}.
\newblock {\em {Logical Methods in Computer Science}}, {Volume 17, Issue 2},
  April 2021.

\bibitem[Sch06]{Schwichtenberg2006}
Helmut Schwichtenberg.
\newblock Constructive analysis with witnesses.
\newblock {\em Proof Technology and Computation. Natio Science Series}, pages
  323--354, 2006.

\bibitem[Sem20]{Semeria2019}
Vincent Semeria.
\newblock {Constructive measure theory}.
\newblock \url{https://github.com/coq-community/corn/pull/88}, 2020.

\bibitem[Spi02]{Spitters:phd}
Bas Spitters.
\newblock {\em Constructive and intuitionistic integration theory and
  functional analysis}.
\newblock PhD thesis, University of Nijmegen, 2002.

\bibitem[Spi05]{Spitters2005}
Bas Spitters.
\newblock Constructive algebraic integration theory without choice.
\newblock In Thierry Coquand, Henri Lombardi, and Marie-Fran{\c{c}}oise Roy,
  editors, {\em Mathematics, Algorithms, Proofs}, number 05021. Dagstuhl, 2005.

\bibitem[Spi06]{Spitters2006}
Bas Spitters.
\newblock Constructive algebraic integration theory.
\newblock {\em Annals of Pure and Applied Logic}, 137(1-3):380--390, 2006.

\end{thebibliography}
\end{document}